\documentclass[lettersize,journal]{IEEEtran}
\usepackage{array}
\usepackage[caption=false,font=normalsize,labelfont=sf,textfont=sf]{subfig}
\usepackage{textcomp}
\usepackage{stfloats}
\usepackage{url}
\usepackage{verbatim}
\usepackage{graphicx}
\usepackage{mathtools}
\usepackage{pifont}
\usepackage{nomencl}
\makenomenclature
\usepackage{etoolbox}
\renewcommand\nomgroup[1]{%
  \item[\bfseries
  \ifstrequal{#1}{V}{Variables}{%
  \ifstrequal{#1}{P}{Parameters}{%
  \ifstrequal{#1}{S}{Sets}{}}}%
]}

\usepackage{cite}

\newcommand{\cmark}{\ding{51}}
\newcommand{\xmark}{\ding{55}}

\usepackage{amsmath,amssymb,amsfonts}
\usepackage{textcomp}
\usepackage{xcolor}
\def\BibTeX{{\rm B\kern-.05em{\sc i\kern-.025em b}\kern-.08em
    T\kern-.1667em\lower.7ex\hbox{E}\kern-.125emX}}
    
\usepackage{amsmath,amsfonts,bm}

\def\eqref#1{equation~\ref{#1}}

\def\1{\bm{1}}

\DeclareMathAlphabet{\mathsfit}{\encodingdefault}{\sfdefault}{m}{sl}
\SetMathAlphabet{\mathsfit}{bold}{\encodingdefault}{\sfdefault}{bx}{n}

\newcommand{\R}{\mathbb{R}}
\newcommand{\C}{\mathbb{C}}

\DeclareMathOperator*{\argmax}{arg\,max}

\usepackage{hyperref}
\usepackage{comment}

\newcommand{\st}{\mathrm{s.t.}}

\usepackage{amsthm}
\newtheorem{theorem}{Theorem}
\newtheorem{lemma}[theorem]{Lemma}
\newtheorem{corollary}[theorem]{Corollary}
\newtheorem{remark}{Remark}

\newtheorem{assumption}{Assumption}

\newtheorem{definition}{Definition}

\usepackage{algorithm}
\usepackage{algorithmicx}
\usepackage{algpseudocode}

\definecolor{tumb}{RGB}{0,101,189}

\newcommand{\revise}[1]{{\color{black}#1}}

\usepackage{ulem}

\colorlet{color1}{blue}
\colorlet{color2}{red!50!black}

\begin{document}

\title{Distributionally Robust Joint Chance-Constrained Optimal Power Flow using Relative Entropy}

\author{Eli Brock*,~Haixiang Zhang*,~Javad Lavaei,~\IEEEmembership{Fellow,~IEEE},~and Somayeh Sojoudi,~\IEEEmembership{Senior Member,~IEEE}
\thanks{\textsuperscript{*}Equal Contribution.}
\thanks{Eli Brock and Somayeh Sojoudi are with the Department of EECS, UC Berkeley (e-mail: eli.brock@bekeley.edu, sojoudi@berkeley.edu). Haixiang Zhang is with the Department of Mathematics, UC Berkeley (e-mail: haixiang\_zhang@bekeley.edu). Javad Lavaei is with the Department of IEOR, UC Berkeley (e-mail: lavaei@berkeley.edu). This work was supported in part by grants from NSF, Noyce Initiative, ONR, AFOSR, and ARO.}
}



\maketitle

\begin{abstract}
Designing robust algorithms for the optimal power flow (OPF) problem is critical for the control of large-scale power systems under uncertainty. The chance-constrained OPF (CCOPF) problem provides a natural formulation of the trade-off between the operating cost and the constraint satisfaction rate. In this work, we propose a new data-driven algorithm for the CCOPF problem, based on distributionally robust optimization (DRO). \revise{We show that the proposed reformulation of the distributionally robust chance constraints is exact, whereas other approaches in the CCOPF literature rely on conservative approximations. We establish out-of-sample robustness guarantees for the distributionally robust solution and prove that the solution is the most efficient among all approaches enjoying the same guarantees.}  We apply the proposed algorithm to the CCOPF problem and compare the performance of our approach with existing methods using simulations on IEEE benchmark power systems.
\end{abstract}


\vspace{-1em}
\section{Introduction}
\label{sec:intro}

Developing resilient algorithms for the optimal power flow (OPF) problem is fundamental to efficient and reliable decision-making in large-scale energy systems. The OPF problem consists of minimizing some objective, including but not limited to generation costs, subject to the physics of the power network as well as additional constraints on power quality, safety, and reliability. Independent system operators solve OPF at several timescales, from hours to minutes ahead of the dispatch time, in order to manage the market and match supply to demand. Traditionally, the primary source of uncertainty in optimal power flow was stochastic loads. This uncertainty was handled through forecasts which were accurate enough that mismatches between supply and demand could be handled in real-time without a significant deviation from nominal network and market conditions. However, given the growing penetration of variable renewable energy (VRE), more sophisticated methods will be necessary to ensure that decisions can be made as efficiently as possible while being robust to large forecast errors.

\revise{Random power injections from VRE forecast error affect the power flow of the network, which appears in the constraints of the OPF problem. To deal with the randomness in the constraints, a suitable definition of constraint satisfaction is required to define the feasible set of the problem. For example,}
the robust optimization (RO) approach was proposed in \cite{jabr_robust_2015} and \cite{louca_robust_2019} to find the worst-case solution, namely, the optimal decision that satisfies all constraints for all possible realizations of the randomness in the system. The RO approach produces the most conservative solution and results in a high operational cost. 

Chance-constrained optimal power flow (CCOPF) allows for a small user-specified probability of violating the constraints in the OPF solution in exchange for a much better operational cost (small violations will later be handled via a real-time control mechanism) Chance-constrained methods avoid the conservativeness associated with the RO approach, which ensures an operating point that is feasible for all possible realizations of a system's forecast errors. Refer to \cite{roald_chance-constrained_2018} and \cite{venzke_convex_2018} for popular formulations of CCOPF.

A challenge for CCOPF is that the true underlying distribution of the random parameters is generally unknown and must be inferred from historical data. \revise{Conventional data-based reformulations of chance constraints include the sample average approximation \cite{pagnoncelli_sample_2009} and the scenario approach \cite{campi_scenario_2009}. Given an allowable violation probability and a tolerance parameter, these approaches lower bound the number of samples required to achieve a given degree of confidence in the probability of satisfying the chance constraints.} The sample average approximation is easily applicable but may lead to a high-variance estimate of the true distribution. The scenario approach is applied for CCOPF in \cite{roald_chance-constrained_2018} and \cite{venzke_convex_2018}. However, the scenario approach is sample-intensive, may be overly conservative, and is often computationally complex. Additionally, more sample-efficient methods allow for samples over longer time horizons (i.e., a day instead of an hour) to be aggregated into a single realization of a random vector, which could reduce bias if forecast errors follow temporal patterns.

Distributionally robust optimization (DRO) addresses the issue of unknown true data-generation distributions by enforcing the chance constraints for all distributions in an \textit{ambiguity set} centered, in the sense of some characteristic feature of probability distributions, around the empirical distribution \cite{rahimian_frameworks_2022}. The idea is that, given enough samples, the true distribution is highly likely to fall inside the ambiguity set. Several papers have applied DRO techniques to OPF or related problems in energy systems. The authors of \cite{lubin_robust_2016,zhang_distributionally_2016,wei_distributionally_2016,xie_distributionally_2018} employ \textit{moment-based} ambiguity sets containing probability distributions with the first and second moments close to those of the empirical distribution. Li \textit{et al.} \cite{li_distributionally_2019} add a unimodality assumption to the moment-based sets to reduce the conservatism. Moment-based ambiguity sets often yield exact tractable reformulations of the chance-constrained program, but they lose information about the true distribution revealed through other features of the data. \textit{Metric-based} ambiguity sets, by contrast, are constructed using measures of distance between probability distributions, most often the Wasserstein metric, and are more expressive. The metric-based approach has the advantage that various statistical consistency and convergence guarantees can be established for DRO estimators \cite{mohajerin_esfahani_data-driven_2018,van_parys_data_2021}. To reformulate the chance constraints as tractable constraints, inner approximations of Wasserstein metric-based ambiguity sets, such as hyper-cubes \cite{duan_distributionally_2018} and polytopes \cite{zhou_linear_2020}, have previously been studied. However, these inner approximations are overly conservative in practice and lead to pessimistic estimations.

\revise{The DRO approaches discussed in the last paragraph are designed for \textit{disjoint} chance constraints,} in which each constraint individually must be satisfied with a given probability. The chance constraints in CCOPF are formulated disjointly for each two-sided constraint \cite{zhou_linear_2020,xie_distributionally_2018,duan_distributionally_2018} or separately for each upper and lower bound \cite{lubin_robust_2016,li_distributionally_2019,zhang_distributionally_2016}.
\textit{Joint} chance constraints, by contrast, require that a solution be feasible, that is, satisfies \textit{all} constraints simultaneously, with a given probability. Given the same violation probability, joint chance constraints are clearly stronger than disjoint chance constraints. \revise{In addition, the joint CCOPF problem is less studied in literature compared to the disjoint counterpart.} Joint chance constraints can be guaranteed by applying the \revise{Bonferroni approximation} to appropriately scaled disjoint chance constraints; see \cite{baker_joint_2019}. However, this approach is highly conservative and does not exploit the potential correlation between random variables in different constraints. Intuitively, when the randomness between constraints is highly correlated, joint chance constraints can be satisfied at a cost that is only slightly higher than that of the chance constraint of a single stochastic constraint. Yang \textit{et al.} \cite{yang_tractable_2022} build on the Bonferroni approach and achieve an inner approximation of a moment-based ambiguity set for the joint case.

The particularly interesting line of work \cite{guo_data-based_2019,guo_data-based_2019-1,poolla_wasserstein_2021,arab_distributionally_2022, ordoudis_energy_2021} is inspired by \cite{mohajerin_esfahani_data-driven_2018}, which provides a reformulation of Wasserstein metric-based DRO problems using conditional value-at-risk (CVaR). The two-part work \cite{guo_data-based_2019}-\cite{guo_data-based_2019-1} is the first to apply the CVaR reformulation to OPF by penalizing constraint violations in the objective function; however, this is not a chance-constrained approach and cannot guarantee the satisfaction of the constraints in any well-defined sense. Ordoudis \textit{et al} \revise{\cite{ordoudis_energy_2021} propose a (nonconvex) CVAR-based inner approximation of the Wasserstein metric-based ambiguity set and proposes an iterative algorithm to solve it.} Poolla \textit{et al.} \cite{poolla_wasserstein_2021} approximate the joint chance constraints using the Bonferroni approximation and reformulate them using CVaR. To achieve the reformulation, the authors use an inner approximation of the ambiguity set via a hyper-rectangle in the parameter space. Arab \textit{et al.} \cite{arab_distributionally_2022} improve on \cite{poolla_wasserstein_2021} by using an ellipsoidal approximation, which reduces the conservativeness by exploiting the correlation between random variables. While the ellipsoidal approximation improves on the hyper-rectangle approximation, the method in \cite{arab_distributionally_2022} remains overly conservative as a consequence of mismatch between the inner approximation and the ambiguity set; see Section \ref{sec:numerical} for numerical illustrations. To address the above issues, we build upon the conference paper \cite{brock_distributionally_2023} tailored to a class of non-convex problems using DRO to study the CCOPF problem. Compared to \cite{brock_distributionally_2023}, we develop strong theoretical results in the context of power systems, and we numerically illustrate the empirical performance of our approach on benchmark IEEE power systems. \revise{Compared with \cite{arab_distributionally_2022,ordoudis_energy_2021}, our approach does not rely on the prior assumption that the ambiguity set can be well approximated by an ellipsoidal or a CVAR approximation. Furthermore, we establish a high-probability guarantee on the constraint satisfaction rate under the \textit{true} distribution and prove that our solution achieves a lower generation cost than any other method with the same guarantee.}

Inspired by \cite{van_parys_data_2021}, we use a relative entropy-based ambiguity set in our DRO formulation and establish stronger theoretical guarantees than those in existing literature.
We implement the algorithms on benchmark OPF problem instances, showcasing the advantages of our new formulation. We summarize our contributions in the following:
%
\begin{itemize}
    \item Instead of the commonly used Wasserstein metric, our DRO formulation utilizes a relative entropy-based ambiguity set. We prove that the relative entropy-based formulation comes with robustness guarantees on out-of-sample performance and moreover admits the \textit{least conservative} DRO solution in the sense that the solution achieves the minimum possible generation cost among all methods with the same robustness guarantees.
    \item \revise{We provide an \textit{exact} reformulation of \textit{joint} distributionally robust chance constraints over the ambiguity set. By comparison, existing works construct an approximation set of the ambiguity set and/or only consider disjoint chance constraints, which makes it challenging to control the trade-off between the efficiency and robustness of the solution.} Furthermore, our reformulation always leads to a feasible problem, while existing approaches cannot guarantee the feasibility.
    \item We empirically compare the performance of our DRO approach with the state-of-the-art approaches in \revise{\cite{arab_distributionally_2022,yang_tractable_2022,ordoudis_energy_2021} on the IEEE 14-bus and 300-bus test cases. We show that our approach is able to find competitive and efficient solutions that asymptotically satisfy the joint chance constraints, while the approximation algorithms in the literature can lead to overly conservative solutions.}
\end{itemize}
\revise{We note that our exact reformulation is designed for the joint chance constraint and does not include the relaxations of OPF models (e.g., the semi-definite relaxation of AC OPF model). The exact reformulation of chance constraints and relaxations of OPF models are discussed in Sections \ref{sec:dro} and \ref{sec:cc-opf}, respectively.
In \cite{van_parys_data_2021}, the authors established the optimality guarantee of a relative entropy-based ambiguity set in the context of minimizing the expected value of an objective function with deterministic constraints. We have adopted the same definition of the ambiguity set and extended the optimality guarantee to the chance constraint case; see Theorems \ref{thm:drpres-1} and \ref{thm:drpres-2} for the theoretical guarantees.  We have modified the theory in \cite{van_parys_data_2021} and established similar optimality guarantees for the chance-constrained setting. 
}

Table \ref{tab:lit_review} summarizes the relevant existing literature on DRO for power systems (most, but not all, of the listed papers focus on OPF) and illustrates our contributions.
\begin{table}[t]
    \centering\caption{\label{tab:lit_review}Comparison of relevant chance-constrained OPF literature}
    \scriptsize
    \begin{tabular}{c|c|c|c|c}
          & chance-constrained & joint & metric-based & exact reformulation \\
          \cite{lubin_robust_2016} & \cmark & \xmark & \xmark & \cmark \\
          \cite{zhou_linear_2020} & \cmark & \xmark & \cmark & \xmark \\
          \cite{xie_distributionally_2018} & \cmark & \xmark & \xmark & \cmark \\
          \cite{duan_distributionally_2018} & \cmark & \xmark & \cmark & \xmark \\
          \cite{zhang_distributionally_2016} & \cmark & \xmark & \xmark & \cmark \\
          \cite{li_distributionally_2019} & \cmark & \xmark & \xmark & \cmark \\
          \cite{yang_tractable_2022} & \cmark & \cmark & \xmark & \xmark \\
          \cite{guo_data-based_2019,guo_data-based_2019-1} & \xmark & & \cmark & \cmark \\
          \cite{ordoudis_energy_2021} & \cmark & \cmark & \cmark & \xmark \\
          \cite{poolla_wasserstein_2021} & \cmark & \cmark & \cmark & \xmark \\
          \cite{arab_distributionally_2022} & \cmark & \cmark & \cmark & \xmark \\
          this work & \cmark & \cmark & \cmark & \cmark
    \end{tabular}
\end{table}
It is worth mentioning that all works in Table \ref{tab:lit_review} except \cite{arab_distributionally_2022} use the common linearized DC approximation of the nonlinear power flow equations, though this approximation is not always coupled to the specific handling of chance constraints. In comparison, we allow for the full AC OPF problem in this work.


\revise{The remainder of the paper is organized as follows. In Section \ref{sec:dro}, we develop a new exact reformulation of general joint chance-constrained problems. In Section \ref{sec:cc-opf}, we introduce chance-constrained optimal power flow and present different models and techniques to arrive at a problem compatible with the general formulation from Section \ref{sec:dro}. Finally, in Section \ref{sec:numerical}, we implement the proposed algorithm to verify the theory and demonstrate the strong empirical performance compared with existing algorithms. We conclude the paper in Section \ref{sec:conclusion}.
}

{\bf Notation:} For every positive integer $n$, we define $[n] \coloneqq \{1,\dots,n\}$. The set of $n$-dimensional integer, real and complex vectors are denoted as $\mathbb{Z}^n$, $\R^n$ and $\C^n$, respectively. Similarly, we use $\R^{m\times n}$ and $\C^{m\times n}$ to denote the set of $m$-by-$n$ real and complex matrices, respectively. 
Let $\mathbf{1}_n$ and $\mathbf{0}_n$ be the vectors with all elements equal to $1$ and $0$, respectively. 
Denote $e_k$ as the $k$-th standard basis vector of $\mathbb{R}^n$. For any two matrices $X,Y\in\R^{m\times n}$, the inner product between them is defined as $\langle X, Y\rangle \coloneqq \mathrm{Tr}(X^TY)$, where $\mathrm{Tr}$ stands for the trace. For each vector $\mathbf{v}\in\R^n$, we say $\mathbf{v} \leq \mathbf{0}_n$ if $\mathbf{v}_k\leq 0$ for all $k\in[n]$. Let $\|\cdot\|$ be the $2$-norm of vectors. 
We say $f(S) = o(S)$ if $\lim_{S\rightarrow \infty} f(S) / S = 0$.

\section{Distributionally Robust Optimization Approach}
\label{sec:dro}


\revise{In this section, we use DRO techniques to develop exact reformulations of chance-constrained optimization problems. We show in the following sections that the proposed reformulation approach can be applied to deal with both AC and DC CCOPF problems; see Section \ref{sec:cc-opf} for more details.}
To preserve the generality of our results, we consider the general objective function and constraint function
\begin{align*}
    g(\mathbf{X}):\R^d \mapsto \R,\quad h(\mathbf{X},\xi):\R^d \times \R^n  \mapsto \R^m,
\end{align*}
where random vector $\xi\in\R^n$ obeys the distribution $\mathbb{P}_0$, and integers $d$ and $m$ are the size of input variable $\mathbf{X}$ and the number of constraints, respectively. In this subsection, we consider the optimization problem with stochastic constraints:
\begin{align}
\label{eqn:obj-general}
    {\min}_{\mathbf{X}\in\R^{d}}~g(\mathbf{X})\quad \mathrm{s.t.}~h(\mathbf{X},\xi) \leq \mathbf{0}_m.
\end{align}
Note that our theory can be extended to the case when randomness $\xi$ also occurs in the objective function $g$ or the feasible set is a convex subset of $\R^d$. We focus on the simpler problem (\ref{eqn:obj-general}) as our objective is to solve the CCOPF problem (\ref{eq:cc-opf}). We make the following assumption:
\begin{assumption}
\label{asp:dro}
The support of $\mathbb{P}_0$ belongs to a compact set $\Xi \subset \R^n$. 
Both functions $g(\cdot)$ and $h(\cdot,\cdot)$ are continuous.
In addition, for every positive integer $S$ and all realizations $\xi_1,\dots,\xi_S\in\R^n$, problem
\[  {\min}_{\mathbf{X}\in\R^{d}}~g(\mathbf{X})\quad \mathrm{s.t.}~h(\mathbf{X},\xi_j)\leq \mathbf{0}_m,~\forall j\in[S] \]
is feasible and has a finite optimal value.
\end{assumption}
%
\noindent Assumption \ref{asp:dro} requires that the constraints can be satisfied regardless of the forecast error $\xi$. This is generally satisfied by real-world OPF problem instances, which have sufficient reserve and transmission capacity to handle renewable forecast errors when properly operated; see the discussion in Section \ref{sec:cc-opf} for more details.

To deal with the stochastic constraint in problem (\ref{eqn:obj-general}), we target finding the minimum-cost solution under the \textit{joint chance constraint}
\begin{align}
\label{eqn:cc-joint-general}
    \mathbb{P}_0\left[ h(\mathbf{X},\xi) \leq \mathbf{0}_m \right] \geq 1 - \epsilon,
\end{align}
where $\epsilon\in(0,1]$ is the pre-specified maximum failing probability.
\begin{remark}
More generally, our results can be extended to the case when the joint constraints are defined by a convex cone
\begin{align}
\label{eqn:cc-joint-cone-0}
    \mathbb{P}_0\left[ \omega^T h(\mathbf{X},\xi) \leq 0,\quad\forall \omega \in \mathcal{W} \right] \geq 1 - \epsilon,
\end{align}
where $\mathcal{W}$ is the convex cone spanned by weight vectors\footnote{A vector $\omega\in\R^m$ is called a weight vector if $\omega \geq \mathbf{0}_m$ and $\mathbf{1}_m^T\omega=1$.} $\omega_1,\dots,\omega_L$. Constraint (\ref{eqn:cc-joint-cone-0}) reduces to the cardinal case (\ref{eqn:cc-joint-general}) when $L=m$ and $\omega_\ell = \mathbf{e}_\ell$ for all $\ell\in[m]$.
\end{remark}

Suppose, as is the case for VRE generation, that the true distribution $\mathbb{P}_0$ is unknown and only limited historical samples may be available. 
Suppose that there are $S$ independently and identically distributed samples, $\xi^1,\dots,\xi^S$, generated from the distribution $\mathbb{P}_0$. We define the empirical distribution of $\xi$ as
\[ \hat{\mathbb{P}}_S := \frac1S {\sum}_{k\in[S]} \delta_{\xi^k}, \]
where $\delta_{\xi}$ is the Dirac measure at $\xi$. 
The goal of the DRO approach is to use the information from the empirical distribution $\hat{\mathbb{P}}_S$ to find robust solutions that satisfy the chance constraint (\ref{eqn:cc-joint-general}) with high probability. Define the ambiguity set
\begin{align*}
    \mathcal{D}_r\left( \mathbb{P} \right) \coloneqq \left\{ \mathbb{P}'\in\mathcal{P} ~|~ I\left(\mathbb{P}, \mathbb{P}' \right) \leq r\right\},\quad \forall \mathbb{P} \in \mathcal{P},
\end{align*}
where $I(\cdot,\cdot)$ is the relative entropy \cite{cover_elements_1999}, $r > 0$ is the radius and $\mathcal{P}$ is the family of Borel distributions with support in $\Xi$. 
The robustness of the DRO solutions is guaranteed by the satisfaction of chance constraints under all distributions in the ambiguity set $\mathcal{D}_r(\hat{\mathbb{P}}_S )$. Other distributional metrics, such as the Wasserstein metric, are also considered in CCOPF literature \cite{poolla_wasserstein_2021,arab_distributionally_2022}. In this work, however, we use the relative entropy due to the strong optimality guarantees it can provide; see Theorems \ref{thm:drpres-1}-\ref{thm:drpres-2} and \cite{cover_elements_1999,van_parys_data_2021}. Intuitively, large deviation theory guarantees that the relative entropy between the true data-generation distribution and the empirical distribution can be bounded by a value that depends on the sample size \cite{cover_elements_1999}. Hence, the true distribution is contained in the ambiguity set with high probability. Critically, the relative entropy-based ambiguity set is the ``smallest'' ambiguity set with such a property \cite{van_parys_data_2021}. 

\revise{\begin{remark}\label{rem:kl_assym}
In DRO literature \cite{calafiore_ambiguous_2007,hong_kullback-leibler_2012,lam_robust_2016,jiang_data-driven_2016}, a similar ambiguity set, defined as
\begin{align*}
    \tilde{\mathcal{D}}_r\left( \mathbb{P} \right) \coloneqq \left\{ \mathbb{P}'\in\mathcal{P} ~|~ I\left(\mathbb{P}', \mathbb{P} \right) \leq r\right\},\quad \forall \mathbb{P} \in \mathcal{P},
\end{align*}
has been widely studied. Unlike $\mathcal{D}_r(\mathbb{P})$, the fixed distribution $\mathbb{P}$ appears in the second argument of $I$. The two ambiguity sets are different because the relative entropy is asymmetric. By placing the perturbed distribution $\mathbb{P}'$ in the first argument, the alternative ambiguity set only includes distributions that are absolutely continuous with respect to $\mathbb{P}$, which may not include $\mathbb{P}_0$ when $\mathbb{P}=\hat{\mathbb{P}}_S$ and therefore restricts the robustness guarantees that can be established. For this reason, we depart from the traditional relative entropy-based DRO literature and study $\mathcal{D}_r$ instead of $\tilde{\mathcal{D}}_r$. See Remark 3 in \cite{van_parys_data_2021} for a more detailed discussion.
\end{remark}}
 we lack knowledge of the true distribution $\mathbb{P}_0$, we approximate the chance constraint (\ref{eqn:cc-joint-general}) with the \textit{distributionally robust chance constraint}
\begin{equation} \label{eqn:dr-cc}
    \inf_{\mathbb{P}'\in\mathcal{D}_r(\hat{\mathbb{P}}_S)} \mathbb{P}'\left[h(\mathbf{X},\xi)\leq\mathbf{0}_m\right] \geq 1-\epsilon.
\end{equation}
\revise{
Intuitively, the true distribution $\mathbb{P}$ belongs to $\mathcal{D}_r(\hat{\mathbb{P}}_S)$ with a high probability in terms of $\epsilon$ and $S$. Hence, the chance constraint (\ref{eqn:cc-joint-general}) holds with high probability if the distributionally robust chance constraint (\ref{eqn:dr-cc}) is satisfied. We prove the claim rigorously in the following theorem.
%
\begin{theorem}
\label{thm:drpres-1}
For all $\epsilon\in(0,1]$ and $r > 0$, if $\mathbf{X}\in\mathbb{R}^d$ satisfies constraint (\ref{eqn:dr-cc}), then it holds that
\begin{equation}
\label{eqn:drpres-bound-1}
    \mathbb{P}_\infty\Big[ h\left(\mathbf{X}, \xi\right) \leq \mathbf{0}_m\Big]\geq 1 - \epsilon - \exp\left[ -rS + o(S) \right],
\end{equation}
where $\mathbb{P}_\infty$ is the probability measure of the sample path space of $\xi$ under distribution $\mathbb{P}_0$.
\end{theorem}

In the regime when the support $\Xi$ is a finite set, we can apply the strong large deviation principle \cite{van_parys_data_2021} and derive the following finite-sample bound in the same way as Theorem \ref{thm:drpres-1}:
\begin{align}
\label{eqn:drpres-bound-finite}
    \mathbb{P}_\infty\Big[ h\left(\mathbf{X}, \xi\right)  \leq \mathbf{0}_m \Big]\geq 1 - \epsilon - (S + 1)^d e^{-rS}.
\end{align}

We include the finite-sample bound (\ref{eqn:drpres-bound-finite}) for theoretical completeness and the forecast error for VRE can have a continuous support. In particular, the high-probability bound (\ref{eqn:drpres-bound-1}) is satisfied by the solution to the distributionally robust chance-constrained problem
\begin{align}
\label{eqn:drpres}
    \hat{\mathbf{X}}_{\epsilon, r,\hat{\mathbb{P}}_S} := \arg\min_{\mathbf{X}\in\R^d}~ &g(\mathbf{X})\quad\st~ \text{chance constraint }(\ref{eqn:dr-cc}).
\end{align}
%
In the following theorem, we show that $\hat{\mathbf{X}}_{\epsilon, r,\hat{\mathbb{P}}_S}$ achieves the minimum cost among all inputs that satisfy the high-probability bound (\ref{eqn:drpres-bound-1}) and are constructed from the empirical distribution $\hat{\mathbb{P}}_S$.
\begin{theorem}
\label{thm:drpres-2}
Suppose that $\tilde{\mathbf{X}}_{\epsilon, r, \mathbb{P}} \in \R^d$ is a quasi-continuous function of distribution $\mathbb{P}$ and that $\tilde{\mathbf{X}}_{\epsilon, r, \hat{\mathbb{P}}_S}$ satisfies constraint (\ref{eqn:drpres-bound-1}). Then, we have 
\[ \mathbb{P}_\infty\left[ g\left(\tilde{\mathbf{X}}_{\epsilon, r, \hat{\mathbb{P}}_S}\right) < g \left(\hat{\mathbf{X}}_{\epsilon, r,\hat{\mathbb{P}}_S} \right) \right] = 0. \]
%
\end{theorem}
}
\revise{
Combined, Theorems \ref{thm:drpres-1} and \ref{thm:drpres-2} establish out-of-sample performance guarantees for the relative entropy-based distributionally robust approximation to the chance constraint (\ref{eqn:cc-joint-general}), and they show that it is optimal among all solutions that satisfy these guarantees.

Although its solution enjoys the attractive properties established above, the DRO problem (\ref{eqn:drpres}) is not in a form that can be handled directly by existing solvers. In Lemma \ref{lem:drpred}, we derive an equivalent form of the chance constraint (\ref{eqn:dr-cc}), which will facilitate a reformulation of (\ref{eqn:drpres}) as a practically solvable problem. First, we define the scalar-valued maximum over the constraints for a given $\mathbf{X}$ and realization of $\xi$
\begin{equation*}
    \bar{h}(\mathbf{X},\xi) := {\max}_{\ell\in[m]}h_\ell(\mathbf{X},\xi)
\end{equation*}
and the worst-case constraint value
\begin{equation*}
    h^*(\mathbf{X}) := {\max}_{\xi\in\Xi}\bar{h}(\mathbf{X},\xi).
\end{equation*}
}
\revise{
\begin{lemma}\label{lem:drpred}
For all $\epsilon \in (0,1]$ and $r > 0$, there exists a $k(\epsilon,r,S)\in[S+1]$ such that (\ref{eqn:dr-cc}) is satisfied if and only if
\begin{equation*}
    {\min}^{k(\epsilon,r,S)}\left\{\bar{h}(\mathbf{X},\xi^j)~|~j\in[S]\right\}\cup\{\bar{h}^*(\mathbf{X)}\} \leq 0,
\end{equation*}
where $\min^k T$ indicates the $k$-th smallest value in set $T$. In addition, the integer $k(\epsilon,r,S)$ is the solution to
\begin{align*}
    \max_{k\in[S], \mathbf{p}\in\R^{S+1}}~ &k\\
    \st~ &{\sum}_{j\in[k]} \mathbf{p}_j \leq 1 - \epsilon,~ \mathbf{1}_{S+1}^T \mathbf{p} = 1,~ \mathbf{p} \geq \mathbf{0}_{S + 1},\\
    &-{\sum}_{j\in[S]} \log(S \mathbf{p}_j) \leq S r.
\end{align*}
\end{lemma}
}

%
%
%
Lemma \ref{lem:drpred} implies that the distributionally robust approximation of the chance constraint (\ref{eqn:dr-cc}) is equivalent to enforcing the $k(\epsilon,r,S)$ least restrictive constraints derived from the samples. 
\revise{\begin{remark}
We note that a similar technique has been utilized in Theorem 8 in \cite{hong_kullback-leibler_2012}, where the threshold $\bar{\beta}$ can be viewed as a counterpart of $1 - k(\epsilon, r, S) / S$. In general, $\bar{\beta}$ takes a different value since they used a different ambiguity set (see Remark \ref{rem:kl_assym}). In addition, the optimization problem for $\bar{\beta}$ (i.e., problem (31) in \cite{hong_kullback-leibler_2012}) is derived from the dual of a functional optimization problem, whose optimal Lagrangian multiplier has a closed form. In comparison, the counterpart of the functional optimization problem in our formulation does not accept a closed-form solution for the multiplier; see problem (29) in \cite{van_parys_data_2021} for the dual problem. Therefore, it is necessary to use our method to compute the value of $k(\epsilon, r, S)$.
\end{remark}}
\revise{
We now present our main result. Lemma \ref{lem:drpred} allows for an exact reformulation that can be applied directly to CCOPF problems:
\begin{corollary}
\label{cor:main}
Problem (\ref{eqn:drpres}) is equivalent to
    \begin{align}
    \label{eqn:drpres-mip}
        \min_{\mathbf{X}\in\R^d, \mathbf{z}\in\mathbb{Z}^S}~ &g(\mathbf{X})\\
        \nonumber\st~ &\text{if }\mathbf{z}_j=0 \text{ then } h(\mathbf{X}, \xi^j) \leq 0, ~\forall j \in [S],\\
        \nonumber&\mathbf{1}_S^T \mathbf{z} \leq S - k(\epsilon,r,S),\quad \mathbf{z}_j \in \{0, 1\},\quad \forall j \in [S].
    \end{align}
\end{corollary}

Problem (\ref{eqn:drpres-mip}) can be formulated as a mixed-integer program (MIP) by implementing the logical indicator constraint using the big-M method (some solvers and algebraic modeling languages will perform this reformulation automatically). Depending on the structure of $g$ and $h$, problem (\ref{eqn:drpres-mip}) may be a mixed-integer linear program (MILP) or a convex/nonconvex mixed-integer nonlinear program (C-MINLP/NC-MINLP). Off-the-shelf solvers exist for all three problem classes, though NC-MINLP solvers such as Bonmin and Juniper typically use local algorithms as heuristics for the relaxed continuous subproblems and therefore cannot guarantee the global optimality.
}
Compared to existing DRO formulations \cite{poolla_wasserstein_2021,arab_distributionally_2022,ordoudis_energy_2021,yang_tractable_2022}, our formulation provides stronger guarantees in the following two senses. First, the DRO solution $\hat{\mathbf{X}}_{\epsilon, r,\hat{\mathbb{P}}_S}$ achieves the minimum possible cost over all robust solutions that satisfy the joint chance constraint (\ref{eqn:drpres-bound-1}). This optimality property arises from the choice of the relative entropy for the ambiguity set, and such property cannot be established by other distributional metrics, although the Wasserstein metric can provide similar high-probability bounds \cite{mohajerin_esfahani_data-driven_2018}. Second, the mixed-integer reformulation (\ref{eqn:drpres-mip}) is \textit{exact}. In contrast, existing literature on distributionally robust CCOPF considered approximations to the ambiguity set. For metric-based ambiguity sets, these are \textit{inner} approximations that may be overly conservative \cite{ordoudis_energy_2021,arab_distributionally_2022,poolla_wasserstein_2021}; see the comparison results in Section \ref{sec:numerical}. 

In practice, it is preferable for the user to first specify $k$ and then compute the optimal $\epsilon$ and $r$ to maximize the right-hand side of (\ref{eqn:drpres-bound-1}). Given $k\in[S]$ and $\epsilon\in[1 - k/S, 1]$, the maximal radius $r$ such that $k(\epsilon,r,S) = k$ is given by
\[ r = -\frac{k}{S}\log\left(\frac{S(1 - \epsilon)}{k}\right) -\frac{S - k}{S}\log\left(\frac{S \epsilon}{S - k}\right), \]
where we define $0\log0 = 0$. Therefore, when the sample size $S$ is sufficiently large, we ignore the $o(S)$ term on the right-hand side of (\ref{eqn:drpres-bound-1}) and solve the maximization problem
\begin{align}
\label{eqn:optimal-eps}
    \epsilon^*_{k, S} := \argmax_{\epsilon\in[1 - k/S, 1]}~ 1 - \epsilon - \frac{S^S}{k^k (S-k)^{S-k}} (1-\epsilon)^k \epsilon^{S-k},
\end{align}
where we define $0^0 = 1$. The solution to the above problem approximately maximizes the right-hand side of (\ref{eqn:drpres-bound-1}) and can be efficiently found by the bisection algorithm.

\section{Chance-constrained Optimal Power Flow}
\label{sec:cc-opf}

\revise{In this section, we consider the chance-constrained optimal power flow (CCOPF) problem. We present multiple versions of the problem associated with different relaxations and linearizations of the power flow equations. While the complexities of the formulations differ, the DRO techniques developed in Section \ref{sec:dro} can be applied to any of them.}

\renewcommand{\nomname}{}
\subsection{Nomenclature}
\vspace{-2em}
\printnomenclature

\nomenclature[S]{$\mathcal{N}=\{1,\dots,n\}$}{Buses}
\nomenclature[S]{$\mathcal{L}=\mathcal{N}\times\mathcal{N}$}{Lines}
\nomenclature[S]{$\mathcal{PQ},\mathcal{PV},\mathcal{V\theta}$}{Load buses, generator buses, and slack buses}
\nomenclature[V]{$\theta\in\mathbb{C}^n$}{Voltage phase angles}
\nomenclature[V]{$P,Q\in\mathbb{R}^n$}{Active and reactive power injections}
\nomenclature[V]{$v\in\mathbb{R}^n$}{Squared voltage magnitudes}
\nomenclature[V]{$\ell\in\mathbb{N}^{n^2}$}{Active power flow}
\nomenclature[P]{$P^D,Q^D\in\mathbb{R}^n$}{Active and reactive loads}
\nomenclature[P]{$\underline{P}^G,\overline{P}^G\in\mathbb{R}^n$}{Lower and upper active generation limits}
\nomenclature[P]{$\underline{Q}^G,\overline{Q}^G\in\mathbb{R}^n$}{Lower and upper active generation limits}
\nomenclature[P]{$\underline{v},\overline{v}\in\mathbb{R}^n$}{upper and lower squared voltage magnitude limits}
\nomenclature[P]{$\overline{\ell}\in\mathbb{R}^{n^2}$}{Active power line flow limit}
\nomenclature[P]{$c_{kd}\in\mathbb{R}$}{$d$th degree cost coefficient for generator at bus $k$}
\nomenclature[P]{$\Omega\in\mathbb{R}^n:\sum_{k\in\mathcal{N}}\Omega_k=1$}{AGC participation factors}
\nomenclature[P]{$\gamma$}{Ratio of reactive power to real power for VRE output}
\vspace{-1em}
\revise{
\subsection{Formulation}
\label{sec:opf}
The CCOPF problem requires choosing an operating point of a power network that is robust to perturbations caused by the real power forecast error, assumed here to be the result of uncertain VRE generation. Consider a network operating at $(P,Q,v,\theta)$ perturbed by a particular forecast error realization $\xi \in \mathbb{R}^n$. The resultant state, denoted using $\tilde{\cdot}$, satisfies
\begin{subequations} \label{eq:error_response}
    \begin{align}
        \label{eq:active_response} & \tilde{P}_k = P_k + \xi_k - \Omega_k{\sum}_{j\in\mathcal{N}}\xi_j, & \forall k\in\mathcal{PQ}\cup\mathcal{PV}, \\
        \label{eq:reactive_response} & \tilde{Q}_k = Q_k + \gamma\xi_k, & \forall k\in\mathcal{PQ}, \\
        \label{eq:voltage_response} & \tilde{v}_k = v_k, & \forall k\in\mathcal{PV}\cup\mathcal{V\theta}, \\
        \label{eq:angle_response} & \tilde{\theta}_{k} = 0, & \forall k\in\mathcal{V\theta}, \\
        \label{eq:gen_power_flow} & f\left(\tilde{P},\tilde{Q},\tilde{v},\tilde{\theta},\tilde{\ell}\right) = 0,
    \end{align}
\end{subequations}
where $\gamma$ is assumed here to be fixed for all VRE generators. To avoid notational clutter, we suppose that all buses have generators and simply set the upper and lower generation limits and AGC factors to zero at a given bus if no generator exists. Equation (\ref{eq:active_response}) models the active perturbations from the forecast error along with the automatic generator control (AGC), in which each generator is responsible for compensating for a fixed proportion of the forecast error (the change in losses is accounted for by the slack bus). The active generation at the slack bus $\mathcal{V\theta}$ is unconstrained to allow for real loss compensation. Equation (\ref{eq:reactive_response}) models reactive perturbations, whereas generator buses are free to adjust their reactive power to maintain a consistent voltage. Equation (\ref{eq:voltage_response}) keeps the voltage magnitude constant at buses with generators. Equation (\ref{eq:angle_response}) fixes the slack bus voltage angle. \revise{In this work, we choose $\Omega$ to be proportional to the generator capacities, which is a common and reasonable choice for a fixed $\Omega$. Our results can be extended to the case where $\Omega$ is set as a decision variable. We keep $\Omega$ fixed to simplify our experiment setting and focus on illustrating the proposed method.}

\revise{Finally, Equation (\ref{eq:gen_power_flow}) encodes the \textit{power flow model}, given here in polar coordinates, which describes the physics of the network. Multiple power flow models are available, which can be categorized as either nonlinear AC power flow models or linear approximate DC power flow models. Some models employ only some of the physical quantities in the arguments of (\ref{eq:gen_power_flow}), in which case the chance constraints associated with those quantities can be dropped.}

Defining $\textbf{w}:=\begin{bmatrix}P & Q & v & \theta & \ell\end{bmatrix}^\intercal\in\mathbb{R}^{4n+n^2}$ and $\tilde{\textbf{w}}$ accordingly, there exists a unique \textit{implicit function} $u(\textbf{w},\cdot)$ of the forecast error $\xi$ such that $\textbf{w}=u(\textbf{w},\mathbf{0}_n)$ and $\tilde{\textbf{w}}=u(\textbf{w},\xi)$ satisfies Equations (\ref{eq:error_response}) for any $\xi$. This is a consequence of the implicit function theorem.

In context, this means that once the operator chooses a dispatch and assigns participation factors $\Omega$, the network state is determined by the forecast error. Chance constraints are introduced in order to ensure satisfaction of limits on line flows, voltage magnitudes, and generator outputs, resulting in the following joint chance-constrained problem:
\begin{subequations} \label{eq:cc-opf}
    \begin{align}
        \min_{\textbf{w}}~ & c(P) \\
        \label{eq:deterministic} \st~ & f(\textbf{w}) = 0, \\
        \label{eq:cc} & \mathbb{P}_0\left\{\begin{aligned}
            & \underline{P_G} \leq u_P(\textbf{w},\xi) + P^D - \xi \leq \overline{P}_G, \\
            & \underline{Q_G} \leq u_Q(\textbf{w},\xi) + Q^D - \gamma\xi \leq \overline{Q}_G, \\
            & \underline{v} \leq u_v(\textbf{w},\xi) \leq \overline{v}, \\
            & -\overline{\ell} \leq u_\ell(\textbf{w},\xi) \leq \overline{\ell}
        \end{aligned}\right\} \geq 1-\epsilon,
    \end{align}
\end{subequations}
where $\mathbb{P}_0$ is the distribution of perturbation $\xi$ and $u_P$, $u_Q$, $u_v$, and $u_\ell$ are the block components of $u$ associated with the subscripted variables in $w$. The objective is the generator cost:
\begin{equation*}
    c(P)\coloneqq {\sum}_{k\in\mathcal{N}}\Big[c_{k0} + c_{k1} (P_k+P_k^D) + c_{k2} \left(P_k+P_k^D\right)^2 \Big].
\end{equation*}
It will be convenient to express the chance constraint in a compact form. Let $\mathcal{CC}$ denote the set of indices of $w$ associated with the chance-constrained quantities $(P,Q,v,\ell)$. With a suitable choice of $\mathbf{A}\in\mathbb{R}^{(6n+n^2)\times n}$ and $\mathbf{b}\in\mathbb{R}^{6n+n^2}$, the chance constraint (\ref{eq:cc}) becomes
\begin{equation} \label{eq:cc-compact}
    \mathbb{P}_0\left[u_{\mathcal{CC}}(\textbf{w},\xi) + \textbf{A}\xi \leq \textbf{b}\right] \geq 1-\epsilon.
\end{equation}
In our formulation of the CCOPF problem (\ref{eq:cc-opf}), Assumption \ref{asp:dro} is satisfied unless, for instance, the reserve capacity of conventional generators is insufficient to compensate for some realizations of the forecast error. In this case, CCOPF may be infeasible; however, it is reasonable to assume that realistic power networks have sufficient capacity to handle any realization of the forecast error if properly dispatched.

\subsection{Approximate linear uncertainty response}
\label{sec:approx_cc}
\revise{In the case of AC models, we cannot yet apply Corollary \ref{cor:main} because problem (\ref{eq:cc-opf}) is a semi-infinite program (SIP) in general due to the lack of an explicit expression for $u(w,\xi)$. The existing literature \cite{roald_chance-constrained_2018} and \cite{venzke_convex_2018} discussed this issue and proposed approximations of the chance constraints. In this work, we adopt the popular linear approximation from \cite{roald_chance-constrained_2018}, although our method can also be applied to the relaxation from \cite{venzke_convex_2018}.}

Following \cite{roald_chance-constrained_2018}, we note that the forecast errors $\xi$ are typically small relative to the loads $P^D$. Hence, we replace $u(\textbf{w},\xi)$ by its first-order approximation at $\xi=\mathbf{0}_n$. More specifically, we define the Jacobian matrix
\[ \textbf{J}_{\textbf{w}} := \frac{\partial}{\partial\xi}u_{\mathcal{CC}}(\textbf{w},\mathbf{0}_n),\quad \forall \mathbf{w}\in\mathbb{R}^{4n+n^2}. \]
Then, the constraint (\ref{eq:cc-compact}) can be approximated by
\begin{equation} \label{eq:cc-approx}
    \mathbb{P}_0\left[\textbf{w}_{\mathcal{CC}} + \left(\textbf{J}_{\textbf{w}}+\textbf{A}\right)\xi \leq \textbf{b}\right] \geq 1-\epsilon.
\end{equation}
The implicit function theorem provides a closed-form expression for the Jacobian $\textbf{J}_{\textbf{w}}$; we provide it here for completness. Write (\eqref{eq:error_response}) in compact form as
\begin{equation*}
    \Psi(\mathbf{w}, \xi) = 0.
\end{equation*}
Then
\begin{equation}\label{eq:ift}
    \textbf{J}_{\textbf{w}} = -\left[\frac{\partial}{\partial \mathbf{w}}\Psi(\mathbf{w},\mathbf{0}_n)\right]^{-1}\frac{\partial}{\partial\xi}\Psi(\mathbf{w},\mathbf{0}_n)
\end{equation}

Replacing the constraint (\ref{eq:cc-compact}) with (\ref{eq:cc-approx}) makes problem (\ref{eq:cc-opf}) a chance-constrained program with a finite number of constraints. Applying the proposed distributionally robust reformulation of the chance constraints, we obtain the mixed-integer nonlinear program
\begin{align} \label{eq:cc-opf-minlp}
        \min_{\textbf{w}}~ & c(P) \\
        \nonumber\st~ & f(\textbf{w}) = 0, \\
        \nonumber& \text{if }\mathbf{z}_j = 0,\text{ then } \textbf{w}_{\mathcal{CC}} + \left(\textbf{J}_{\mathbf{w}}+\textbf{A}\right)\xi^j \leq \textbf{b},\quad \forall j\in[S], \\
        \nonumber& \mathbf{1}_S^T \mathbf{z} \leq S - k(\epsilon,r,S),\quad \mathbf{z}_j \in \{0, 1\},\quad \forall j \in [S].
\end{align}
Unlike the SIP (\ref{eq:cc-opf}), problem (\ref{eq:cc-opf-minlp}) is a MINLP in general and can be handled by off-the-shelf solvers. If $f$ encodes a linear power flow model such as DC-OPF, then the first-order approximation of the system response is exact and problem (\ref{eq:cc-opf-minlp}) is a MILP. On the other hand, if $f$ encodes an AC power flow model, then existing solvers are heuristics and cannot guarantee optimality.

\subsection{DC power flow formulation}
\label{sec:dc-ccopf}

First, we apply the techniques from Sections \ref{sec:opf} and \ref{sec:approx_cc} to the approximate linear DC power flow model, which is given in the context of (\ref{eq:error_response}) by
\begin{equation*}
    f(P,\ell) := \begin{bmatrix}
        \sum_{k\in\mathcal{N}}P_k \\
        \ell - \Phi P
    \end{bmatrix}.
\end{equation*}
This is the power transfer distribution factor (PTDF)-based DC-OPF formulation, where $\Phi$ is the PTDF matrix and $\ell$ is taken to represent (directed) real power flow. Since the model is linear, the system response functions can be computed in closed form as
\begin{align*}
    u_P(P,\ell,\xi) &= P + (I-\Omega\mathbf{1}_n^\intercal)\xi, \\
    u_\ell(P,\ell,\xi) &= \ell + \Phi(I-\Omega\mathbf{1}_n^\intercal)\xi.
\end{align*}
Corollary \ref{cor:main} can now be applied directly.

\subsection{AC power flow formulation}
\label{sec:ac-ccopf}

We use the quadratic form of the AC power flow equations. Defining $\mathbf{Y}_k$, $\bar{\mathbf{Y}}_k$, , $\mathbf{Y}_{jk}$, and $\mathbf{M}_k$ as in \cite{lavaei_zero_2012}, define
\begin{align*}
    \hat{f}(P,Q,v,\ell,\textbf{X}) := \begin{bmatrix}
        P_k - \textbf{X}^T\mathbf{Y}_k\textbf{X}, & \forall k\in\mathcal{N} \\
        Q_k - \textbf{X}^T\bar{\mathbf{Y}}_k\textbf{X}, & \forall k\in\mathcal{N} \\
        \ell_{jk} - \textbf{X}^T\mathbf{Y}_{jk}\textbf{X}, & \forall (j,k)\in\mathcal{L} \\
        v_k - \textbf{X}^\intercal\mathbf{M}_k\textbf{X}, & \forall k\in\mathcal{N}
    \end{bmatrix},
\end{align*}
Make the change-of-variables 
\begin{equation} \label{eq:cov}
    \textbf{X}:=\begin{bmatrix}
    \sqrt{v}\odot\cos{\theta}\\ \sqrt{v}\odot\sin\theta
\end{bmatrix}.
\end{equation} 
Then, the AC power flow model can be written as 
\begin{equation*}
    f(\textbf{w}) = \hat{f}\left(P,Q,v,\ell,\textbf{X}\right) = \hat{f}(\hat{\textbf{w}}),
\end{equation*}
where we define $\hat{\textbf{w}}:=(P,Q,v,\ell,\textbf{X})$.
Since we have made a change-of-variables, we define the implicit function in terms of $\hat{\textbf{w}}$ that characterizes the system response to forecast errors $\xi$. By the implicit function theorem, there exists a unique function $\hat{u}(\hat{\textbf{w}},\xi)$ of the forecast error $\xi$ such that $\hat{\textbf{w}}=\hat{u}(\hat{\textbf{w}},\mathbf{0}_n)$ and $(\tilde{P},\tilde{Q},\tilde{v},\tilde{\ell},\tilde{\textbf{X}})=\hat{u}(\hat{\textbf{w}},\xi)$ satisfies Equations (\ref{eq:active_response})-(\ref{eq:voltage_response}) along with
\begin{subequations}
    \begin{align}
        \label{eq:angle_response_cov} \tilde{\textbf{X}}_{k+n} &= 0, \quad \forall k\in\mathcal{V\theta}, \\
        \label{eq:gen_power_flow_cov} \hat{f}(\tilde{P},\tilde{Q},\tilde{v},\tilde{\ell},\tilde{\textbf{X}}) &= 0.
    \end{align}
\end{subequations}
As before, we define
\[ \hat{\textbf{J}}_{\hat{\textbf{w}}} := \frac{\partial}{\partial\xi}u_{\mathcal{CC}}(\hat{\textbf{w}},\mathbf{0}_n),\quad \forall \mathbf{w}\in\mathbb{R}^{4n+n^2}. \]
which can be computed similarly to (\eqref{eq:ift}). Note that under the change-of-variables mapping (\ref{eq:cov}) between $\hat{\textbf{w}}$ and $\textbf{w}$, it holds that $\hat{u}_\mathcal{CC}(\hat{\textbf{w}},\xi) = u_\mathcal{CC}(\textbf{w},\xi)$. Given this equivalence, problem (\ref{eq:cc-opf-minlp}) becomes
\begin{align} \label{eq:acopf-reformulated}
        \min_{\hat{\textbf{w}}}~ & c(P) \\
        \nonumber\st~ & \hat{f}(\hat{\textbf{w}}) = 0, \\
        \nonumber& \text{if }\mathbf{z}_j = 0,\text{ then } \hat{\textbf{w}}_{\mathcal{CC}} + \left[\hat{\mathbf{J}}_{\hat{\mathbf{w}}}+\textbf{A}\right]\xi^j \leq \textbf{b},\quad\forall j\in[S], \\
        \nonumber& \mathbf{1}_S^T \mathbf{z} \leq S - k(\epsilon,r,S),\quad \mathbf{z}_j \in \{0, 1\},\quad \forall j \in [S].
\end{align}

\revise{Problem (\ref{eq:acopf-reformulated}) is now in a form that can be handled by local-search NC-MINLP solvers. Since these solvers are heuristics, one may wish to apply relaxations and approximations to (\ref{eq:acopf-reformulated}). In the following subsections, we explain two techniques which, for networks satisfying certain sufficient conditions, recast (\ref{eq:acopf-reformulated}) as a C-MINLP that can be solved globally by appropriate solvers.

\subsubsection{Fixed-point algorithm} \label{sec:fixed-point}
One source of nonconvexity in (\ref{eq:cc-opf-minlp}) is that the Jacobian $\hat{\mathbf{J}}_{\mathbf{\textbf{w}}}$ is nonconvex in $\hat{\textbf{w}}$. This can be partially addressed by applying the fixed-point iteration in \cite[Sec. V. B.]{roald_chance-constrained_2018}. The pseudo-code of this heuristic algorithm is provided in Algorithm \ref{alg:fixed-point}. Intuitively, if the initialization is close to the solution, the fixed-point iteration enjoys fast convergence. 

In most applications, the forecast errors are small relative to the forecast power injections and thus, our approximation scheme is considerably accurate in the following sense:
\begin{enumerate}
    \item The first-order approximation is acceptable under a wide range of operating conditions.
    \item The robust solution to the chance-constrained problem is expected to be not too far from the deterministic solution (i.e., the solution with $\xi=\mathbf{0}_n$). 
\end{enumerate}
As a consequence, although there is no convergence guarantee, the fixed-point iteration exhibits efficient and robust convergence in practice; see the numerical experiments in Section \ref{sec:numerical} and \cite{roald_chance-constrained_2018}.

A number of methods can be applied to solve step \ref{step:dro}. If the method proposed here is applied, step \ref{step:dro} becomes
\begin{align} \label{eq:cc-opf-minlp-fp}
        w_{t+1}\gets&\arg\min_{\textbf{w}}~  c(P) \\
        \nonumber\st~ & f(\textbf{w}) = 0, \\
        \nonumber& \text{if }\mathbf{z}_j = 0,\text{ then } \textbf{w}_{\mathcal{CC}} + \left(\textbf{J}_{\mathbf{w}_t}+\textbf{A}\right)\xi^j \leq \textbf{b},\quad \forall j\in[S], \\
        \nonumber& \mathbf{1}_S^T \mathbf{z} \leq S - k(\epsilon,r,S),\quad \mathbf{z}_j \in \{0, 1\},\quad \forall j \in [S].
\end{align}
Other authors, such as \cite{arab_distributionally_2022}, have taken an even simpler approach by using the Jacobian evaluated at the deterministic solution obtained from assuming a zero forecast error (i.e., $\mathbf{w}_0$ in Algorithm \ref{alg:fixed-point}).

\subsubsection{Convex relaxation} \label{sec:relaxation}

The second source of nonconvexity arises from the nonlinearity of $\hat{f}$. This may be addressed by applying the well-known semi-definite relaxation from \cite{lavaei_zero_2012}. The relaxation is known to be exact for a large class of networks, including many IEEE benchmark cases.

Other convex relaxations of AC power flow models exist and can be applied along with our method. These include the second-order cone relaxation for the QCQP formulation and the second-order cone relaxation of the branch flow equations, both of which are exact for radial networks under mild conditions \cite{low_convex_2014}. \revise{We note that it is necessary to apply the reduction method in \cite{low_convex_2014} to recover the rank-$1$ solution from the SDP solution, which usually has a higher rank. In our experiments, we are always able to recover a rank-$1$ solution and the results presented are computed using the recovered rank-$1$ voltage vector instead of the SDP solution matrix.}

If both the fixed-point iteration and a relaxation are applied, problem (\ref{eq:cc-opf-minlp}) reduces to a C-MINLP problem. Either method described here may also be used independently to reduce the nonconvexity of (\ref{eq:cc-opf-minlp}) before passing to a local-search solver.}
}
\begin{algorithm}[t]
\caption{Fixed-point iteration for joint CCOPF problem.}
\label{alg:fixed-point}
    \begin{algorithmic}[1]
    \State {\bf Input:} tolerance $\mu$, maximum violation probability $\epsilon$.
    \State {\bf Output:} robust solution $\textbf{w}$.
        \State \textbf{Initialization:} 
        \begin{align*}
            \textbf{w}_0\gets
            \arg\min_{\textbf{w}}~ & c(P) \\
            \st~ & f(\textbf{w}) = 0, \quad  \textbf{w}_\mathcal{CC} \leq \textbf{b}.
        \end{align*}
        \For{ $t=0,1,\dots$} 
            \State \label{step:dro} Update $\mathbf{w}_{t+1}$ to be the (approximate) solution of
            \begin{equation}
                \begin{aligned}
                    \min_{\textbf{w}}~ & c(P) \\
                    \st~ & f(\textbf{w}) = 0, \\
                    & \mathbb{P}_0\left[\mathbf{w}_\mathcal{CC} + \left(\mathbf{J}_{\mathbf{w}}+\mathbf{A})\right)\xi\right] \geq 1-\epsilon
                \end{aligned}
            \end{equation}
        \State \textbf{If} {$D\left(\mathbf{w}_{t+1} - \mathbf{w}_{t}\right) \leq \eta$,} \textbf{return} $\mathbf{w}_{t+1}$.
        \Statex\Comment{$D$ is some appropriate measure of distance.}
        \EndFor
    \end{algorithmic}
\end{algorithm}
\revise{
\section{Demonstration on IEEE Test Cases}
\label{sec:numerical}

In this section, we apply the proposed DRO approach to IEEE benchmark power systems and show the empirical performance of our method over existing methods in the literature. To help the readers better understand the operation of our approach, we first demonstrate the algorithm on a simple
$14$-bus example and explain each step in detail. Then, we consider both the DC and AC models for the CCOPF problem. We focus on the $14$-bus and $300$-bus systems, and we compare the performance with several existing DRO algorithms. For each case, the satisfaction rate, the generation cost and the running time are exhibited to verify our theory and illustrate the advantage of our developed approach.
}
\revise{
\subsection{Tutorial}
To build intuition, we first present a walkthrough of the proposed method on a simple example. Consider the IEEE 14-bus test case with VRE generators installed at buses 2 and 3, each with a forecast output of 20 MW. The VRE forecast error follows a multivariate normal distribution, constructed as described in Section \ref{sec:setup} (except without the clipping). For simplicity, we consider the DC model where only real generator outputs and line flows appear in the chance-constraints. Line flows in this example are set sufficiently high as to always be inactive, so we only consider generator output. The 14-bus case has 6 generators, located at buses 1, 2, 3, 6, and 8. Generators 1 and 2 have linear costs of \$20/MW, while the others have linear costs of \$40/MW; the quadratic components of the costs are relatively small. Assuming no forecast error and solving the deterministic OPF, the optimal dispatch is to supply the full net load of the network using generator 1, with all other generator outputs set to zero. However, under the AGC scheme described in Section \ref{sec:opf}, this decision will result in the zero-output generators carrying insufficient downward reserves in case of unexpectedly high VRE output.

Suppose that a system operator has access to 100 historical samples of the outputs of the VRE generators at buses 2 and 3 and wants to ensure that the remaining generators carry sufficient downward reserves at least $90\%$ of the time. To apply the proposed method, the system operator must enforce the constraints under all but a small number of the historical scenarios. Specifically, it must find the smallest $k$ such that $\epsilon_{k,100}^*$ is no more than $0.10$, where $\epsilon_{k,S}^*$ is given in (\ref{eqn:optimal-eps}). In this case, we have $\epsilon_{97,100}^*=0.109$ and $\epsilon_{98,100}^*=0.0924$. Therefore, to find the optimal generation coast, the operator enforces the constraints under the $98$ \textit{best-case} samples. Figure \ref{fig:tutorial_samples} shows the 100 samples and the worst-case pair selected by solving the mixed-integer reformulation.
\begin{figure}
    \centering
    \includegraphics[width=0.75\linewidth]{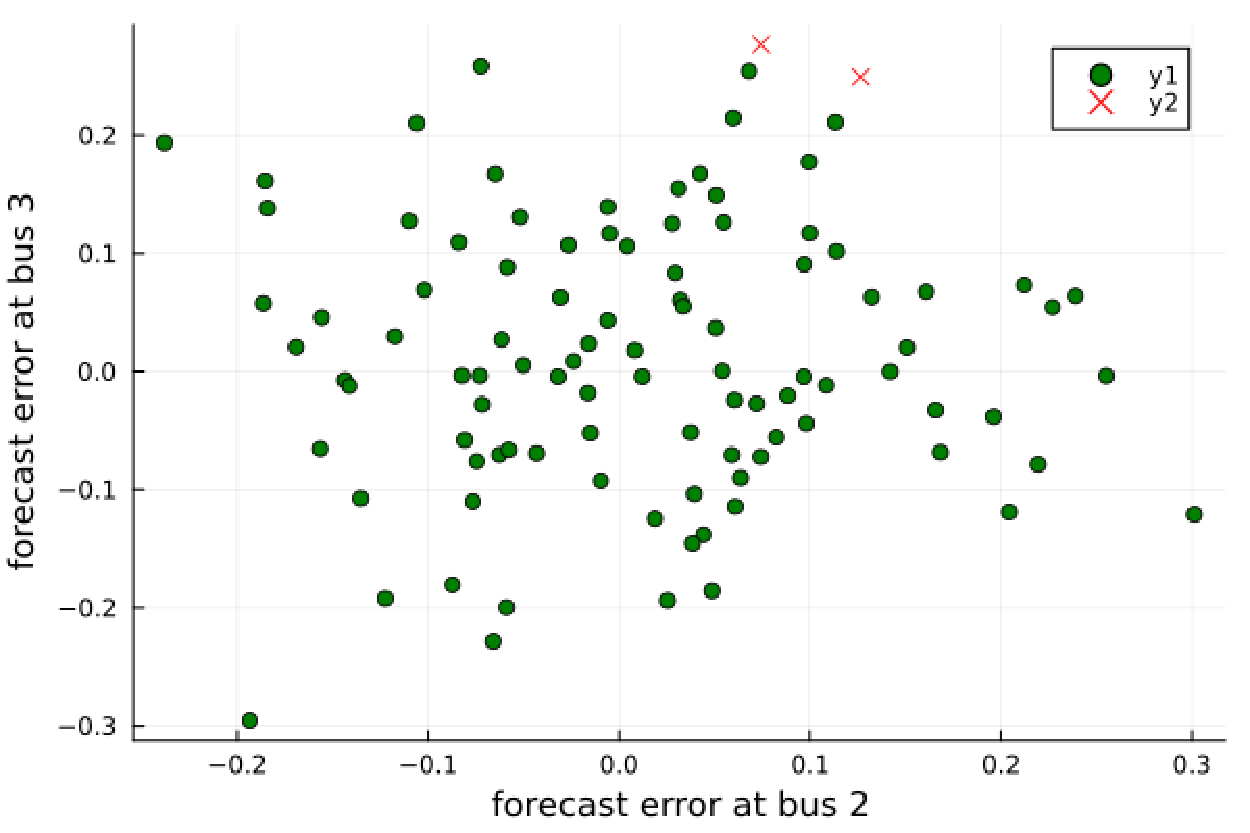}
    \caption{The 100 historical samples with the two worst-case samples in red.}
    \label{fig:tutorial_samples}
\end{figure}

As expected, the worst-case set of two samples for this example are those where VRE output exceeded the forecast, requiring higher downward reserves. Note that despite only requiring $90\%$ constraint satisfaction, $98\%$ of the samples must be enforced to obtain the out-of-sample performance guarantee from Theorem \ref{thm:drpres-1}. As more samples are available, fewer ``extra'' samples will be required as the empirical distribution will better approximate the true distribution. Practically, enforcing the $98$ best-case samples means that generators 2-6 produce a small amount of power to allow for downward generation response. This results in a higher total generation cost, but is much more robust to errors. Figure \ref{fig:tutorial_distribution} shows the output distribution of generator 2 after the AGC redispatch for 1000 test samples, drawn independently of the historical samples from the same generation distribution. The two subplots show the distributions for the zero-error deterministic setpoint and the distributionally robust KL setpoint. As expected, only a small number of cases result in an output below the lower bound of zero for the KL case, while the deterministic setpoint violates the chance constraints approximately half of the time.
\begin{figure}
    \centering
    \includegraphics[width=0.8\linewidth]{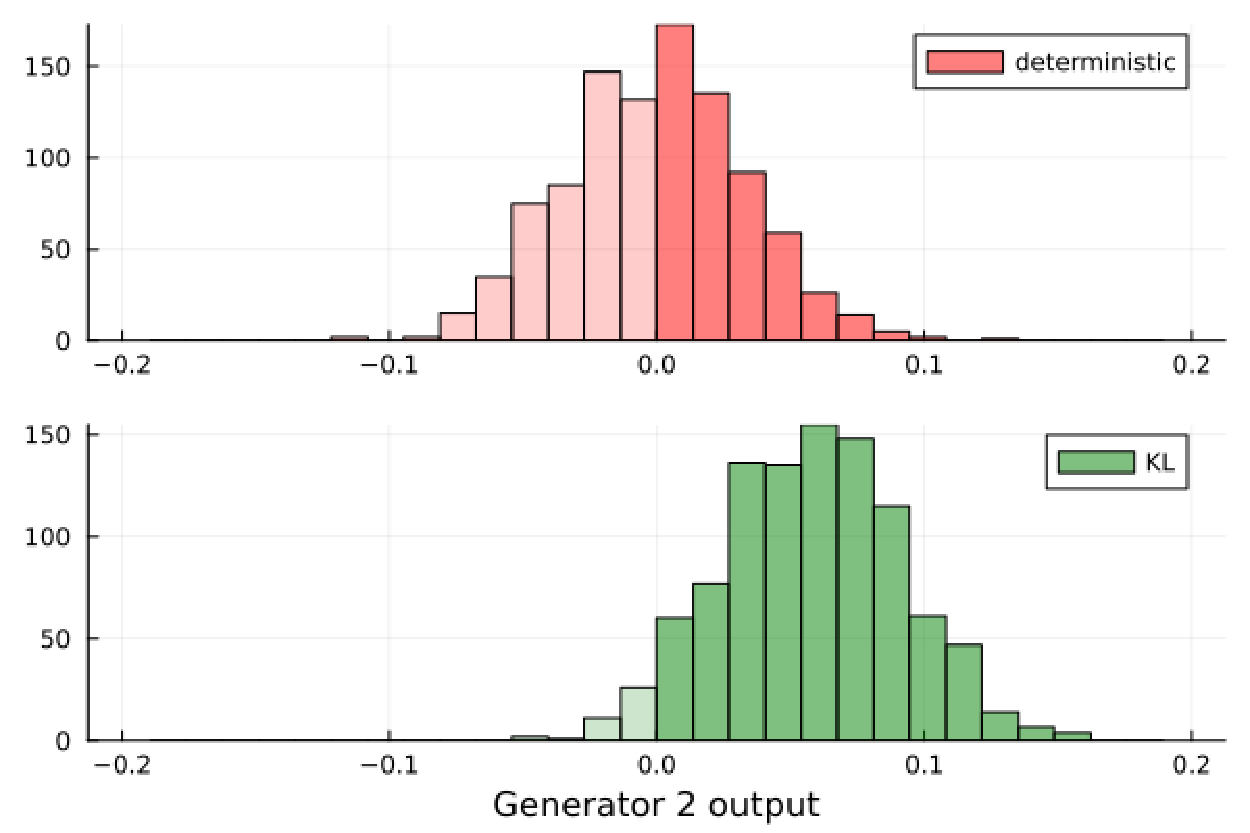}
    \caption{The output distribution of generator 2 under the proposed method compared to the zero-error dispatch.}
    \label{fig:tutorial_distribution}
\end{figure}
All data-based chance-constrained methods will accomplish the robustness in a similar way. However, we prove in Theorem \ref{thm:drpres-2} that among all methods which enjoy the out-of-sample performance guarantee in Theorem \ref{thm:drpres-1}, our proposed method is the most efficient in the sense that it achieves the optimal generation cost. In the context of this example, this means that generators 2-6 increase their outputs \textit{just enough} to maintain sufficient downward reserves $90\%$ of the time (with high probability), but not more.
}

\revise{
\subsection{Overview of benchmark methods}
Recall that the original problem of interest is (\ref{eq:cc-opf}). If a DC power flow model is used, then the joint chance constraint (\ref{eq:cc}) includes a finite number of linear constraints; we denote this as the DC problem. For AC power flow models, (\ref{eq:cc}) has an infinite number of constraints and we instead consider the approximate problem with chance constraint (\ref{eq:cc-approx}). Although (\ref{eq:cc-approx}) is affine in $\xi$, it remains nonconvex due to both the dependence of $\mathbf{J}_\mathbf{w}$ on $w$ and the nonlinear power flow equations (\ref{eq:deterministic}). We name the problem with (\ref{eq:cc-approx}) as the AC problem. The nonconvexity in the AC problem associated with $\mathbf{J}_\mathbf{w}$ can be mitigated by applying the fixed-point algorithm discussed in Section \ref{sec:fixed-point}; we call this the AC fixed-point approach (AC-FP). While the subproblems (step \ref{step:dro} in Algorithm \ref{alg:fixed-point}) in AC-FP remain nonconvex because of (\ref{eq:deterministic}), the chance-constraints are affine, allowing for the application of appropriate reformulations from the literature. For networks satisfying certain sufficient conditions, AC-FP can be convexified through the semi-definite relaxation discussed in Section \ref{sec:relaxation}, giving the relaxed fixed-point approach, which we name as SDP-FP.

We compare the proposed method (denoted as the KL method) against three existing methods from the literature for handling joint chance constraints using limited samples: the Wasserstein-based ellipsoidal method (W-E) from \cite{arab_distributionally_2022}, the optimal CVAR approximation (OPT-CVAR) from \cite[Sec. 4.3]{ordoudis_energy_2021}, and the moment-based optimized Bonferroni approximation (M-OBA) method from \cite[Sec. 3.A.]{yang_tractable_2022}. As discussed in Section \ref{sec:intro}, most other methods from the literature either consider \textit{individual} chance constraints or, in the case of the sample average or scenario methods, require much more samples than we presume are available to system operators. These methods are not suitable for comparison as they do not share the problem setting considered in this work. The methods selected for comparison constitute the state-of-the-art in the CCOPF space to the best of the authors' knowledge.

The authors of \cite{yang_tractable_2022} used the M-OBA to approximate joint affine chance-constraints as a single nonconvex deterministic constraint. Additionally, multiple approximations and convexifications of this constraint are presented. However, solving the nonconvex OBA problem locally seems to have similar performance in practice to the approximations as seen in \cite[Sec. V]{yang_tractable_2022}. Therefore, we solve the OBA formulation locally as the representative method of \cite{yang_tractable_2022}.

The KL approach could be directly applied to the AC problem (without the fixed-point iteration or the convex relaxation). However, given the reliable performance of Algorithm \ref{alg:fixed-point} in practice and to avoid corrupting our comparisons by potential non-global solutions, we solve the SDP-FP formulation in our simulations for the AC problem.
}
\revise{
\subsection{Experimental setup} \label{sec:setup}
We demonstrate our method on the joint CCOPF problem (\ref{eq:cc-opf}) for the IEEE 14- and 300-bus test cases. The starting system parameters are taken from the case data in MATPOWER 8.0 \cite{zimmerman_matpower_2011,zimmerman_matpower_2020} and are prepossessed to allow for matrix-based analysis using PowerModels.jl's \texttt{make\_basic\_network} utility \cite{coffrin_powermodels_2018}. Among other things, this preprocessing step adds reasonable line limits even when they do not exist in the MATPOWER files.

We install VRE generators at the same buses and with the same outputs as the simulations in \cite{yang_tractable_2022}, except we triple the forecasts for the 300-node case. The participation factors $\Omega$ are set proportional to the maximum output of each conventional generator.

Again following \cite{yang_tractable_2022}, the samples $\xi$ are drawn from a multivariate Gaussian distribution $\mathcal{N}(\mu,\Sigma)$ with $\mu=0$ and
\begin{align*}
    \Sigma_{ii} &= \zeta p_i,\quad \forall i,\quad \text{and }\Sigma_{ij} = \rho_{ij}\sqrt{\Sigma_{ii}}\sqrt{\Sigma_{jj}},\quad \forall i \neq j,
\end{align*}
where $p_i$ is the forecast output of the VRE generator at bus $i$ and $\rho_{ij}>0$ induces some correlation between VRE outputs, as will happen in practice. We set $\zeta=0.05$ for the 14-bus case, $\zeta=0.1$ for the 300-bus case, and $\rho_{ij}=\rho=0.2$ for both cases. Unlike in \cite{yang_tractable_2022}, we clip the $i$-th component of each sample to $[-p_i,2p_i]$ to model a VRE generator operating at one-third of its capacity. We set $\gamma=0.1$ so that a small amount of reactive power is included in the forecast error. For each case, we generate $200$ historical samples and compute the out-of-sample joint chance constraint violation frequency on $10000$ independent test samples.

All optimization problems are solved through the JuMP interface \cite{lubin_jump_2023}, where QPs and MIQPs are solved by Gurobi \cite{gurobi_optimization_llc_gurobi_2023}, SDPs are solved by Clarabel \cite{goulart_clarabel_2024} and MISDPs are solved with Pajarito \cite{coey_outer_2020} with SCIP \cite{bolusani_scip_2024} as the outer approximation solver and Clarabel as the conic solver. Simulations are run on in Julia 1.11.0 on a 6-core 2.60 GHz laptop computer running Windows 11.
}
\revise{
\subsection{Comparison with OPT-CVAR and M-OBA on 14-bus DC model}
We first present results for the DC model of the 14-bus network. Comparisons are made with OPT-CVAR and M-OBA, both of which were introduced for DC problems. In this context, the proposed KL method needs to solve a MIQP problem. The OPT-CVAR solves a sequence of QPs and LPs, and the M-OBA involves solving a NC-NLP. Recall that the KL method is most naturally applied by choosing the minimum $k$ such that $\epsilon_{k,S}^*$ as computed in (\ref{eqn:optimal-eps}) is at least the desired maximum violation probability $\epsilon$. The other methods in the literature are applied by choosing $\epsilon$ directly. Therefore, we compare the methods by testing on each $k$ from $180$ to $200$, in intervals of $2$, and feeding the associated $\epsilon_{k,S}^*$ to the other methods. For each $k$, we compute the generation cost (efficiency) of the solution, the out-of-sample violation frequency and the solve time. For OPT-CVAR, \footnote{We assume an unbounded support $\Xi$ for the implementation of OPT-CVAR as including the associated variables ($\gamma$) proves extremely memory-intensive.}, we choose a Wasserstein radius of $\rho=10^{-3}$ and a minimum relative improvement threshold of $\eta=10^{-5}$. The results are shown in Figure \ref{fig:14-dc}, where results of the KL method are shown as a solid line while comparative methods are dotted.
\begin{figure}
    \centering
    \includegraphics[width=0.7\linewidth]{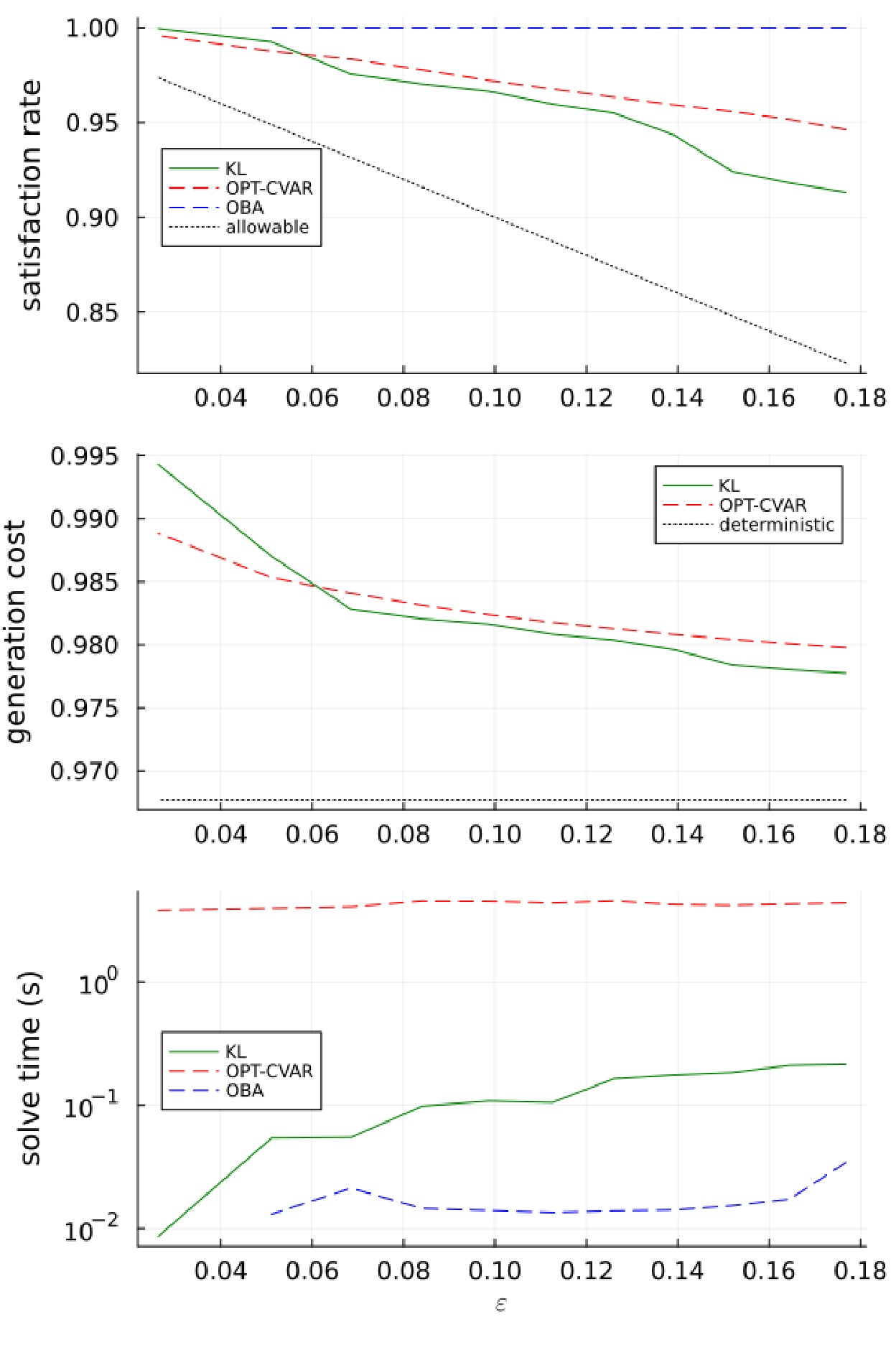}
    \caption{Performance on the 14-bus DC network.}
    \label{fig:14-dc}
\end{figure}

The dotted line in the first subplot is simply $y=1-x$, representing the desired minimum out-of-sample violation frequency. We expect the satisfaction rate of CC solutions to remain above this line. This subplot gives the \textit{joint} constraint violation; that is, if any single constraint is violated for a given sample, the joint chance constraint for sample is considered to be violated. While KL and OPT-CVAR behave as expected, OBA produces overconservative solutions with zero constraint violation probability regardless of $\epsilon$ (OBA was infeasible for the smallest value of $\epsilon$ tested in this experiment). We note that since the DRO solution is constructed based on the empirical distribution $\hat{\mathbb{P}}_S$, there exists a gap between the estimated satisfaction rate and the actual satisfaction rate; see the error term in Theorem \ref{thm:drpres-1}. If the number of samples $S$ becomes larger, the empirical distribution $\hat{\mathbb{P}}_S$ will correlate better with $\mathbb{P}_0$ and the satisfaction rate will converge to the pre-set value $1-\epsilon$.

The second subplot, showing the generation cost, is normalized to the approximate robust optimization (RO) cost, which is computed by enforcing the constraints of 10000 training samples. As DRO is designed as a compromise between stochastic programming and RO, distributionally robust methods should generally achieve more efficient solutions than RO at the cost of some robustness; that is, one should expect the costs to be less than 1 in this subplot. As the KL method reformulates the chance constraints exactly, this is guaranteed to be true. Other methods, however, use approximations to the chance constraint and may be more conservative than RO in general. Indeed, OBA proves to be extremely conservative, with normalized costs greater than 1.006 for all values of $\epsilon$ and is omitted from the generation cost subplot so as not to distort the trends of KL and OPT-CVAR. The deterministic solution, achieved by solving the (non-chance-constrained) problem without forecast errors, is also included as another benchmark. The deterministic solution is \textit{not} robust and will violate the chance constraints very often, but its cost represents a loose lower bound on the achievable efficiency gains of DRO relative to RO.

Both KL and OPT-CVAR exhibit the efficiency-robustness trade-off as expected from DRO. As $\epsilon$ grows, their out-of-sample constraint satisfaction rate dips, though it remains well above the $1-\epsilon$ threshold. As a consequence, they close a significant portion of the distance between the RO and deterministic solutions. In these simulations, KL is slightly more efficient than OPT-CVAR for larger $\epsilon$ and slightly less efficient for smaller $\epsilon$. That CVAR may achieve lower costs than KL is not a counterexample to the optimality guarantee in Theorem \ref{thm:drpres-2}, as this guarantee only applies to methods featuring the specific out-of-sample performance guarantees in Theorem \ref{thm:drpres-1}; CVAR, as a Wasserstein-based method, does not fall into this category.

The final plot gives the solution time on a log scale. The OBA runs the fastest, followed by KL, which is an order of magnitude faster than OPT-CVAR. Notice that the solution time of KL is lower for small values of $\epsilon$ since the number of feasible integer solutions for a given $k$ is $\binom{S}{k}$, which grows at an exponential rate with respect to $k$.

\subsection{Comparison with W-E on 14-bus AC model}
We now test the proposed method on the SDP-FP formulation using the AC power flow model presented in Section \ref{sec:ac-ccopf} for the 14-bus network. We consider the SDP-FP formulation instead of the AC or AC-FP problems since it results in a convex problem and we are able to guarantee the global optimality; the semi-definite relaxation is known to be exact for the 14-bus network \cite{lavaei_zero_2012}. However, we emphasize that our method could be applied to the AC or AC-FP problems with minor changes.

We compare with W-E, as it was originally presented for (relaxed) AC models. The physical solution to the original problem is recovered from the matrix solution of the relaxed problem by fixing the optimal values of $\begin{bmatrix}v_{\mathcal{V}\Theta\cup\mathcal{PV}} & P_\mathcal{PV}\end{bmatrix}$ and running a power flow (one could also apply \cite[Cor. 2]{lavaei_zero_2012}). All results presented here are calculated using the recovered physical solution, not the relaxed solution. In this context, the KL needs to solve an MISDP and W-E involves solving a pair of SDPs.

The KL method applies the fixed-point algorithm (Algorithm \ref{alg:fixed-point}) to reduce the nonconvexity in $\mathbf{J}_{\textbf{w}}$. In our experiments, we set $\eta=10^{-4}$ and $D$ is taken to be the 2-norm between the subvectors $\begin{bmatrix}v_{\mathcal{V}\Theta\cup\mathcal{PV}} & P_\mathcal{PV}\end{bmatrix}$, which together fully specify the power flow solution. For all experiments, the fixed-point algorithm converged in 3 iterations or fewer.
%
In contrast, in \cite{arab_distributionally_2022}, the authors simply linearized around the deterministic solution rather than apply the fixed-point algorithm. The Wasserstein radius ($\rho$ in \cite{arab_distributionally_2022}) is set to $10^{-6}$. This is the smallest (least conservative) radius considered in \cite{arab_distributionally_2022}. As seen in Figure \ref{fig:14_ac}, KL is considerably less conservative even with such an aggressive choice of Wasserstein radius. For small values of $\epsilon$, W-E produced a solution more conservative than RO. In terms of running time, W-E runs faster than KL since KL involves a MISDP and W-E only includes solving SDPs. In summary, we observe that in the two 14-bus experiments, the metric-based methods (KL, OPT-CVAR, and W-E) outperform the moment-based method (OBA) in general.

\begin{figure}
    \centering
    \includegraphics[width=0.7\linewidth]{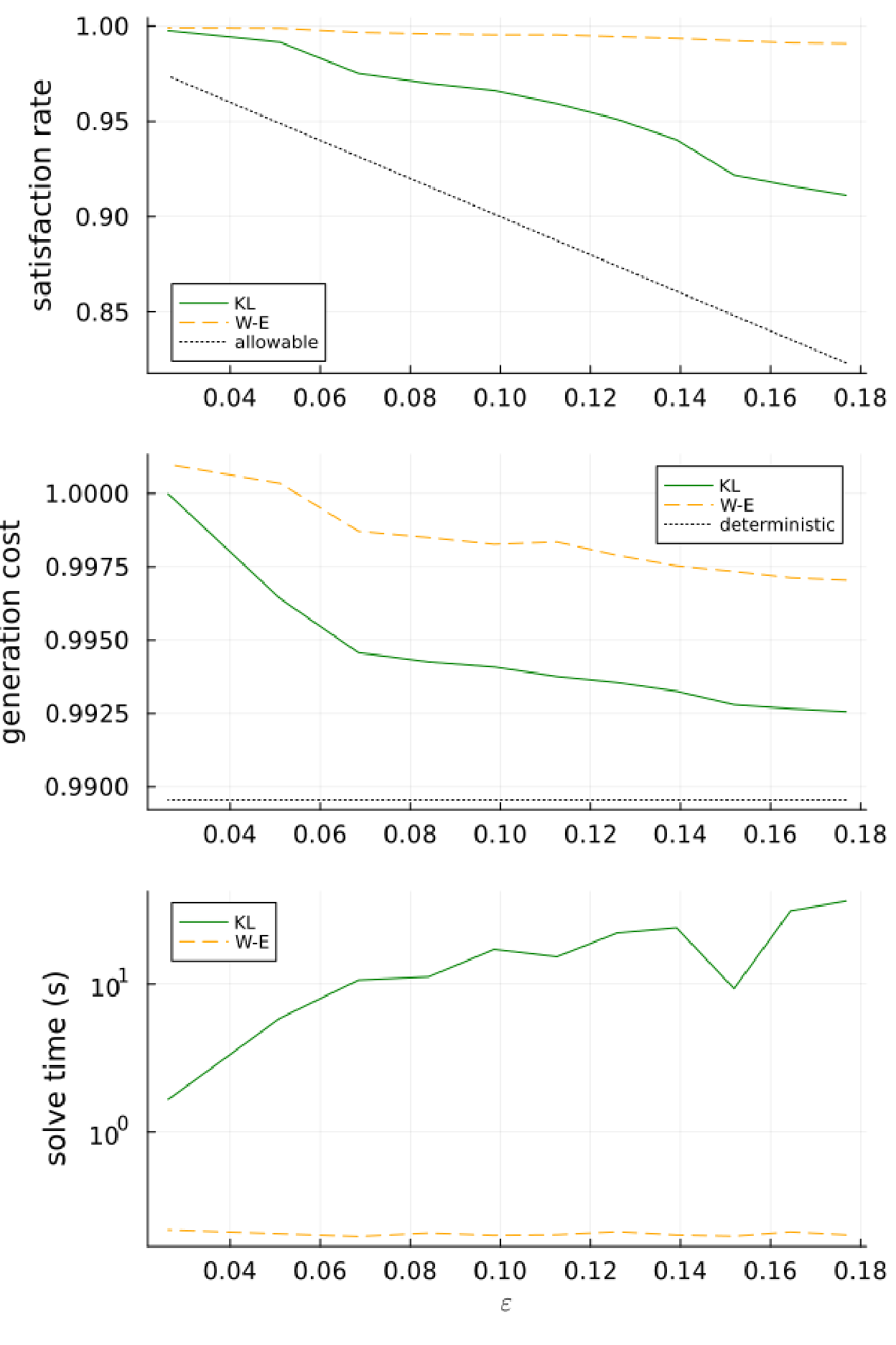}
    \caption{Performance on the 14-bus AC network.}
    \label{fig:14_ac}
\end{figure}

\subsection{Results for the 300-bus network}
To demonstrate the scalability of the proposed method, we demonstrate the performance of KL on the 300-bus network using the DC power flow model presented in Section \ref{sec:dc-ccopf} and compare with the OPT-CVAR method. We generate $S=300$ training samples and $10000$ test samples. The sweep of $k$ is from $270$ to $300$ in increments of $3$.

The KL method performs on the 300-node network as it does on the 14-bus network, with the violation rate dropping as permitted by $\epsilon$ and bridging the gap in efficiency between the RO and deterministic solutions. As for the 14-bus network, KL is slightly more conservative than OPT-CVAR for small $\epsilon$ and slightly less conservative for large $\epsilon$. For all $\epsilon$, KL ran for the larger network with more samples in less than 10 seconds on a personal computer, approximately an order of magnitude faster than OPT-CVAR. As OPF problems are typically solved offline, this runtime is promising for larger-scale implementations.

Notice that, though the values of $k$ range from $90\%$ to $100\%$ of $S$ in all three experiments, the values of $\epsilon$ on the x-axis are lower for when 300 samples are used rather than 200. This is because $\epsilon$ better approximates $k/S$ as the sample size $S$ grows. Intuitively, as more historical data is available, the empirical distribution more closely approximates the data-generation distribution, allowing one to enjoy the same high-probability bound while enforcing the constraints of fewer samples.

\begin{figure}
    \centering
    \includegraphics[width=0.7\linewidth]{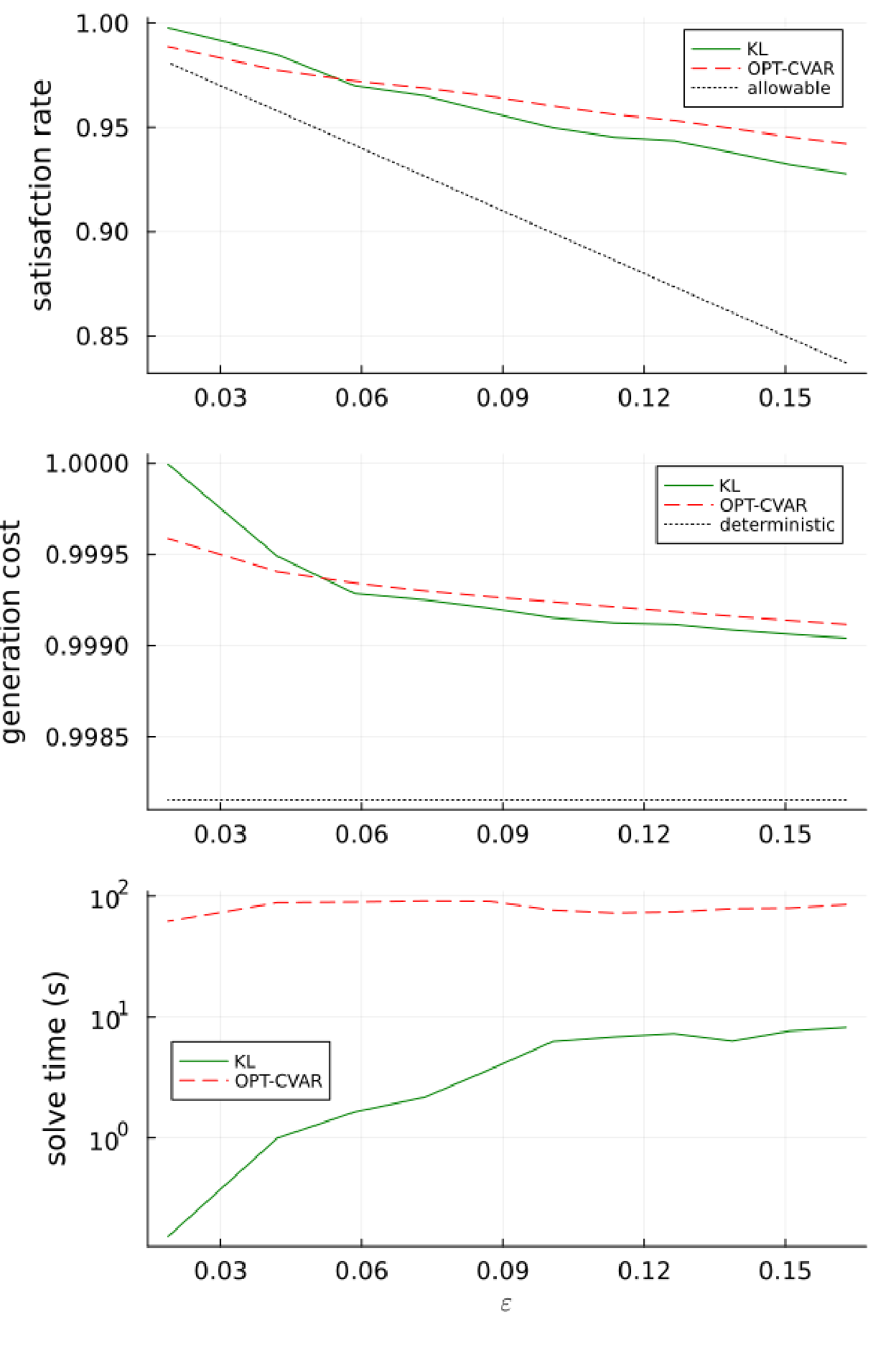}
    \caption{Performance on 300-bus DC network.}
    \label{fig:14}
\end{figure}
}

\vspace{-0.5em}
\revise{
\section{Conclusion}
\label{sec:conclusion}

This work proposes a novel distributionally robust optimization technique for chance-constrained OPF problems with large random VRE forecast errors. Between the theoretical and empirical results, the advantages of the proposed method can be broadly categorized as follows:
\begin{enumerate}
    \item \textbf{Efficiency}. Theorem \ref{thm:drpres-2} guarantees that the KL solution is the most efficient in terms of the generation cost among all solutions achieving the out-of-sample performance guarantees established in Theorem \ref{thm:drpres-1}. In our experiments, KL is considerably more efficient than the state-of-the-art methods in \cite{yang_tractable_2022} and \cite{arab_distributionally_2022}, which are sometimes more conservative than even the robust optimization method. Due to the exact reformulation in \ref{cor:main}, KL can never be more conservative than RO. For the computational efficiency, KL achieves performance similar to that of OPT-CVAR with faster runtime in our simulations.
    \item \textbf{Stability and ease of use.} Unlike other methods presented in Section \ref{sec:numerical}, KL does not require the tuning of any hyperparameters, particularly an ambiguity set radius, which is unintuitive for practitioners. In addition, KL does not ask the user to make implementation decisions as in \cite{yang_tractable_2022}, which includes five different variants. Moreover, due to Corollary \ref{cor:main}, KL is guaranteed to be feasible as long as the mild Assumption \ref{asp:dro} is satisfied while methods that rely on conservative approximations may become infeasible for small $\epsilon$. The reliability of the method is particularly important for applications with large-scale networks and AC power flow models.
\end{enumerate}
}


\vspace{-1em}
\section*{Acknowledgments}
This material is based upon work supported in part by the U. S. Army Research Laboratory and the U. S. Army Research Office under grant number W911NF2010219. It was also supported by ONR, AFOSR, NSF, and the UC Noyce Initiative. The authors would like to thank Julie Mulvaney Kemp for comments and suggestions.
\vspace{-1em}

\bibliographystyle{ieeetr}
\bibliography{ref.bib}

\begin{thebibliography}{10}

\bibitem{jabr_robust_2015}
R.~A. Jabr, S.~Karaki, and J.~A. Korbane, ``Robust {Multi}-{Period} {OPF}
  {With} {Storage} and {Renewables},'' {\em IEEE Transactions on Power
  Systems}, vol.~30, pp.~2790--2799, Sept. 2015.

\bibitem{louca_robust_2019}
R.~Louca and E.~Bitar, ``Robust {AC} {Optimal} {Power} {Flow},'' {\em IEEE
  Transactions on Power Systems}, vol.~34, pp.~1669--1681, May 2019.

\bibitem{roald_chance-constrained_2018}
L.~Roald and G.~Andersson, ``Chance-{Constrained} {AC} {Optimal} {Power}
  {Flow}: {Reformulations} and {Efficient} {Algorithms},'' {\em IEEE
  Transactions on Power Systems}, vol.~33, no.~3, pp.~2906--2918, 2018.

\bibitem{venzke_convex_2018}
A.~Venzke, L.~Halilbasic, U.~Markovic, G.~Hug, and S.~Chatzivasileiadis,
  ``Convex {Relaxations} of {Chance} {Constrained} {AC} {Optimal} {Power}
  {Flow},'' {\em IEEE Transactions on Power Systems}, vol.~33, pp.~2829--2841,
  May 2018.

\bibitem{pagnoncelli_sample_2009}
B.~K. Pagnoncelli, S.~Ahmed, and A.~Shapiro, ``Sample {Average} {Approximation}
  {Method} for {Chance} {Constrained} {Programming}: {Theory} and
  {Applications},'' {\em Journal of Optimization Theory and Applications},
  vol.~142, pp.~399--416, Aug. 2009.

\bibitem{campi_scenario_2009}
M.~C. Campi, S.~Garatti, and M.~Prandini, ``The scenario approach for systems
  and control design,'' {\em Annual Reviews in Control}, vol.~33, pp.~149--157,
  Dec. 2009.

\bibitem{rahimian_frameworks_2022}
H.~Rahimian and S.~Mehrotra, ``Frameworks and {Results} in {Distributionally}
  {Robust} {Optimization},'' {\em Open Journal of Mathematical Optimization},
  vol.~3, pp.~1--85, 2022.

\bibitem{lubin_robust_2016}
M.~Lubin, Y.~Dvorkin, and S.~Backhaus, ``A {Robust} {Approach} to {Chance}
  {Constrained} {Optimal} {Power} {Flow} {With} {Renewable} {Generation},''
  {\em IEEE Transactions on Power Systems}, vol.~31, pp.~3840--3849, Sept.
  2016.

\bibitem{zhang_distributionally_2016}
Y.~Zhang, S.~Shen, and J.~Mathieu, ``Distributionally {Robust}
  {Chance}-{Constrained} {Optimal} {Power} {Flow} with {Uncertain} {Renewables}
  and {Uncertain} {Reserves} {Provided} by {Loads},'' {\em IEEE Transactions on
  Power Systems}, pp.~1--1, 2016.

\bibitem{wei_distributionally_2016}
W.~Wei, F.~Liu, and S.~Mei, ``Distributionally {Robust} {Co}-{Optimization} of
  {Energy} and {Reserve} {Dispatch},'' {\em IEEE Transactions on Sustainable
  Energy}, vol.~7, pp.~289--300, Jan. 2016.

\bibitem{xie_distributionally_2018}
W.~Xie and S.~Ahmed, ``Distributionally {Robust} {Chance} {Constrained}
  {Optimal} {Power} {Flow} with {Renewables}: {A} {Conic} {Reformulation},''
  {\em IEEE Transactions on Power Systems}, vol.~33, pp.~1860--1867, Mar. 2018.

\bibitem{li_distributionally_2019}
B.~Li, R.~Jiang, and J.~L. Mathieu, ``Distributionally {Robust}
  {Chance}-{Constrained} {Optimal} {Power} {Flow} {Assuming} {Unimodal}
  {Distributions} {With} {Misspecified} {Modes},'' {\em IEEE Transactions on
  Control of Network Systems}, vol.~6, pp.~1223--1234, Sept. 2019.

\bibitem{mohajerin_esfahani_data-driven_2018}
P.~Mohajerin~Esfahani and D.~Kuhn, ``Data-driven distributionally robust
  optimization using the {Wasserstein} metric: performance guarantees and
  tractable reformulations,'' {\em Mathematical Programming}, vol.~171,
  pp.~115--166, Sept. 2018.

\bibitem{van_parys_data_2021}
B.~P.~G. Van~Parys, P.~M. Esfahani, and D.~Kuhn, ``From {Data} to {Decisions}:
  {Distributionally} {Robust} {Optimization} {Is} {Optimal},'' {\em Management
  Science}, vol.~67, pp.~3387--3402, June 2021.

\bibitem{duan_distributionally_2018}
C.~Duan, W.~Fang, L.~Jiang, L.~Yao, and J.~Liu, ``Distributionally {Robust}
  {Chance}-{Constrained} {Approximate} {AC}-{OPF} {With} {Wasserstein}
  {Metric},'' {\em IEEE Transactions on Power Systems}, vol.~33,
  pp.~4924--4936, Sept. 2018.

\bibitem{zhou_linear_2020}
A.~Zhou, M.~Yang, M.~Wang, and Y.~Zhang, ``A {Linear} {Programming}
  {Approximation} of {Distributionally} {Robust} {Chance}-{Constrained}
  {Dispatch} {With} {Wasserstein} {Distance},'' {\em IEEE Transactions on Power
  Systems}, vol.~35, pp.~3366--3377, Sept. 2020.

\bibitem{baker_joint_2019}
K.~Baker and A.~Bernstein, ``Joint {Chance} {Constraints} in {AC} {Optimal}
  {Power} {Flow}: {Improving} {Bounds} {Through} {Learning},'' {\em IEEE
  Transactions on Smart Grid}, vol.~10, pp.~6376--6385, Nov. 2019.

\bibitem{yang_tractable_2022}
L.~Yang, Y.~Xu, H.~Sun, and W.~Wu, ``Tractable {Convex} {Approximations} for
  {Distributionally} {Robust} {Joint} {Chance}-{Constrained} {Optimal} {Power}
  {Flow} {Under} {Uncertainty},'' {\em IEEE Transactions on Power Systems},
  vol.~37, pp.~1927--1941, May 2022.

\bibitem{guo_data-based_2019}
Y.~Guo, K.~Baker, E.~Dall'Anese, Z.~Hu, and T.~H. Summers, ``Data-{Based}
  {Distributionally} {Robust} {Stochastic} {Optimal} {Power} {Flow}—{Part}
  {II}: {Case} {Studies},'' {\em IEEE Transactions on Power Systems}, vol.~34,
  pp.~1493--1503, Mar. 2019.

\bibitem{guo_data-based_2019-1}
Y.~Guo, K.~Baker, E.~Dall'Anese, Z.~Hu, and T.~H. Summers, ``Data-{Based}
  {Distributionally} {Robust} {Stochastic} {Optimal} {Power} {Flow}—{Part}
  {I}: {Methodologies},'' {\em IEEE Transactions on Power Systems}, vol.~34,
  pp.~1483--1492, Mar. 2019.

\bibitem{poolla_wasserstein_2021}
B.~K. Poolla, A.~R. Hota, S.~Bolognani, D.~S. Callaway, and A.~Cherukuri,
  ``Wasserstein {Distributionally} {Robust} {Look}-{Ahead} {Economic}
  {Dispatch},'' {\em IEEE Transactions on Power Systems}, vol.~36,
  pp.~2010--2022, May 2021.

\bibitem{arab_distributionally_2022}
A.~Arab and J.~E. Tate, ``Distributionally {Robust} {Optimal} {Power} {Flow}
  via {Ellipsoidal} {Approximation},'' {\em IEEE Transactions on Power
  Systems}, 2022.
\newblock Publisher: IEEE.

\bibitem{ordoudis_energy_2021}
C.~Ordoudis, V.~A. Nguyen, D.~Kuhn, and P.~Pinson, ``Energy and reserve
  dispatch with distributionally robust joint chance constraints,'' {\em
  Operations Research Letters}, vol.~49, pp.~291--299, May 2021.

\bibitem{brock_distributionally_2023}
E.~Brock, H.~Zhang, J.~Mulvaney~Kemp, J.~Lavaei, and S.~Sojoudi,
  ``Distributionally {Robust} {Optimization} for {Nonconvex} {QCQPs} with
  {Stochastic} {Constraints},'' in {\em 2023 62th {IEEE} {Conference} on
  {Decision} and {Control} ({CDC})}, pp.~1--7, IEEE, 2023.

\bibitem{cover_elements_1999}
T.~M. Cover, {\em Elements of information theory}.
\newblock John Wiley \& Sons, 1999.

\bibitem{calafiore_ambiguous_2007}
G.~C. Calafiore, ``Ambiguous {Risk} {Measures} and {Optimal} {Robust}
  {Portfolios},'' {\em SIAM Journal on Optimization}, vol.~18, pp.~853--877,
  Jan. 2007.

\bibitem{hong_kullback-leibler_2012}
J.~L. Hong and Z.~Hu, ``Kullback-{Leibler} {Divergence} {Constrained}
  {Distributionally} {Robust} {Optimization},'' Nov. 2012.

\bibitem{lam_robust_2016}
H.~Lam, ``Robust {Sensitivity} {Analysis} for {Stochastic} {Systems},'' {\em
  Mathematics of Operations Research}, vol.~41, pp.~1248--1275, Nov. 2016.

\bibitem{jiang_data-driven_2016}
R.~Jiang and Y.~Guan, ``Data-driven chance constrained stochastic program,''
  {\em Mathematical Programming}, vol.~158, pp.~291--327, July 2016.

\bibitem{lavaei_zero_2012}
J.~Lavaei and S.~H. Low, ``Zero {Duality} {Gap} in {Optimal} {Power} {Flow}
  {Problem},'' {\em IEEE Transactions on Power Systems}, vol.~27, pp.~92--107,
  Feb. 2012.
\newblock Conference Name: IEEE Transactions on Power Systems.

\bibitem{low_convex_2014}
S.~H. Low, ``Convex {Relaxation} of {Optimal} {Power} {Flow}—{Part} {I}:
  {Formulations} and {Equivalence},'' {\em IEEE Transactions on Control of
  Network Systems}, vol.~1, pp.~15--27, Mar. 2014.

\bibitem{zimmerman_matpower_2011}
R.~D. Zimmerman, C.~E. Murillo-Sanchez, and R.~J. Thomas, ``{MATPOWER}:
  {Steady}-{State} {Operations}, {Planning}, and {Analysis} {Tools} for {Power}
  {Systems} {Research} and {Education},'' {\em IEEE Transactions on Power
  Systems}, vol.~26, pp.~12--19, Feb. 2011.

\bibitem{zimmerman_matpower_2020}
R.~D. Zimmerman and C.~E. Murillo-Sánchez, ``{MATPOWER},'' Oct. 2020.
\newblock Language: en.

\bibitem{coffrin_powermodels_2018}
C.~Coffrin, R.~Bent, K.~Sundar, Y.~Ng, and M.~Lubin, ``{PowerModels}. {JL}:
  {An} {Open}-{Source} {Framework} for {Exploring} {Power} {Flow}
  {Formulations},'' in {\em 2018 {Power} {Systems} {Computation} {Conference}
  ({PSCC})}, pp.~1--8, June 2018.

\bibitem{lubin_jump_2023}
M.~Lubin, O.~Dowson, J.~D. Garcia, J.~Huchette, B.~Legat, and J.~P. Vielma,
  ``{JuMP} 1.0: recent improvements to a modeling language for mathematical
  optimization,'' {\em Mathematical Programming Computation}, vol.~15,
  pp.~581--589, Sept. 2023.

\bibitem{gurobi_optimization_llc_gurobi_2023}
{Gurobi Optimization, LLC}, ``Gurobi {Optimizer} {Reference} {Manual},'' 2023.

\bibitem{goulart_clarabel_2024}
P.~J. Goulart and Y.~Chen, ``Clarabel: {An} interior-point solver for conic
  programs with quadratic objectives,'' May 2024.
\newblock arXiv:2405.12762.

\bibitem{coey_outer_2020}
C.~Coey, M.~Lubin, and J.~P. Vielma, ``Outer approximation with conic
  certificates for mixed-integer convex problems,'' {\em Mathematical
  Programming Computation}, vol.~12, pp.~249--293, June 2020.

\bibitem{bolusani_scip_2024}
S.~Bolusani, ``The {SCIP} {Optimization} {Suite} 9.0,'' Feb. 2024.

\bibitem{wolsey_integer_1999}
L.~A. Wolsey and G.~L. Nemhauser, {\em Integer and combinatorial optimization},
  vol.~55.
\newblock John Wiley \& Sons, 1999.

\bibitem{neubrunn_quasi-continuity_1988}
T.~Neubrunn, ``Quasi-continuity,'' {\em Real Analysis Exchange}, vol.~14,
  no.~2, pp.~259--306, 1988.

\bibitem{prokhorov_convergence_1956}
Y.~V. Prokhorov, ``Convergence of {Random} {Processes} and {Limit} {Theorems}
  in {Probability} {Theory},'' {\em Theory of Probability \& Its Applications},
  vol.~1, pp.~157--214, Jan. 1956.

\end{thebibliography}

\newpage
\begin{appendix}

\subsection{Omitted Proofs in Section \ref{sec:dro}}

In this section, we first introduce the necessary notation and then present the proofs of the theoretical results in Section \ref{sec:dro}.
Define the $\alpha$-quantile
\begin{align*} 
    q_\alpha(F, \mathbb{P}) := \sup \left\{ q ~|~  \mathbb{P}\left[ F(\xi) \leq q \right] \leq \alpha \right\}
\end{align*}
for all $\alpha\in[0,1]$, function $F(\cdot):\R^n \mapsto \R$ and distribution $\mathbb{P}\in\mathcal{P}$. Then, the chance constraint (\ref{eqn:cc-joint-general}) can be equivalently written as
\begin{align}
\label{eqn:cc-joint-cone}
    \nonumber &\mathbb{P}_0\left[ h_\ell(\mathbf{X},\xi) \leq 0,\quad\forall \ell \in [m] \right] \geq 1 - \epsilon\\
    \nonumber \iff &\mathbb{P}_0\left[ \bar{h}(\mathbf{X}, \xi) \leq 0 \right] \geq 1 - \epsilon\\
    \iff &q_{1 - \epsilon}\left[\bar{h}(\mathbf{X}, \cdot), \mathbb{P}_0\right] \leq 0,
\end{align}
%
Adopting language from \cite{van_parys_data_2021}, we first introduce the distributionally robust predictor of the $\alpha$-quantile.
\begin{definition}[Distributionally Robust Predictor]
For all $\epsilon\in(0,1]$, $r > 0$, $\mathbf{X}\in\R^d$ and $\mathbb{P}\in\mathcal{P}$, the distributionally robust predictor is defined as
\begin{align*}
    \hat{q}_{1 - \epsilon, r, \mathbb{P}}(\mathbf{X}) := \sup_{\mathbb{P}'\in\mathcal{D}_r(\mathbb{P})}~ q_{1-\epsilon}\left[\bar{h}(\mathbf{X},\cdot), \mathbb{P}'\right].
\end{align*}
%
\end{definition}
Intuitively, the distributionally robust predictor is the \textit{worst-case} $\alpha$-quantile over all distributions in the relative entropy ball $\mathcal{D}_r(\mathbb{P})$. In Lemma \ref{lem:drpred}, we show that the distributionally robust predictor is either a quantile of $\bar{h}_{\mathbf{X}}$ under the empirical distribution $\hat{\mathbb{P}}_S$ or the maximum value $h^*(\textbf{X})$. We restate Lemma \ref{lem:drpred} for the reader's convenience.
%
%
\begin{lemma}[Restatemeant of Lemma \ref{lem:drpred}]
For all $\epsilon \in (0,1]$ and $r > 0$, there exists an integer $k(\epsilon, r, S) \in [S + 1]$ such that
\begin{align*}
    \hat{q}_{1 - \epsilon, r, \hat{\mathbb{P}}_S}(\mathbf{X}) = \bar{h}_{k(\epsilon, r, S), \hat{\mathbb{P}}_S}\left( \mathbf{X} \right),\quad \forall \mathbf{X} \in \R^d.
\end{align*}
%
When there is no confusion, we denote for the notational simplicity $k \coloneqq k(\epsilon, r, S)$ and
\begin{align*}\bar{h}_{k}(\mathbf{X}) &\coloneqq {\min}^{k(\epsilon,r,S)}\left\{\bar{h}(\mathbf{X},\xi^j)~|~j\in[S]\right\}\cup\{\bar{h}^*(\mathbf{X)}\}. \end{align*}
%
\end{lemma}
%
\begin{proof}
We first show that in the definition of predictor $\hat{q}_{1 - \epsilon, r, \hat{\mathbb{P}}_S}$, the supremum can be restricted to the set of distributions in $\mathcal{D}_r (\hat{\mathbb{P}}_S)$ that are absolutely continuous with respect to $\hat{\mathbb{P}}_S$ except on the set 
\[ \Xi^*(\mathbf{X}) := \left\{\xi~|~\bar{h}_{\mathbf{X}}(\xi) = h^*_{\mathbf{X}} \right\}. \]
The proof is the same as that of Lemma 2 of \cite{van_parys_data_2021} except the bound on the expectation, i.e., the second last inequality in the proof. To deal with this issue, we only need to prove that for all $\mathbf{X}\in\R^d$, $p\in[0,1]$, $\xi^*\in\Xi^*(\mathbf{X})$, and $\mathbb{P}_c,\mathbb{P}_\perp\in\mathcal{P}$ such that $\mathbb{P}_c \ll \hat{\mathbb{P}}_S$ and $\mathbb{P}_\perp \perp \mathbb{P}_c$\footnote{For distributions $\mathbb{P},\mathbb{P}'\in\mathcal{P}$, we use $\mathbb{P} \ll \mathbb{P}'$ and $\mathbb{P} \perp \mathbb{P}'$ to denote the case when $\mathbb{P}$ is absolutely continuous and singular with respect to $\mathbb{P}'$, respectively.}, it holds that
\begin{align}
\label{eqn:inequ-quantile}
    q_{1 - \epsilon}\left(\bar{h}_{\mathbf{X}}, \mathbb{P}'\right) \geq q_{1 - \epsilon}\left(\bar{h}_{\mathbf{X}}, \mathbb{P}''\right),
\end{align}
where 
\[ \mathbb{P}' \coloneqq p \cdot \mathbb{P}_c + (1 - p) \cdot \delta_{\xi^*},\quad \mathbb{P}'' \coloneqq p \cdot \mathbb{P}_c + (1 - p) \cdot \mathbb{P}_\perp. \]
Let $F'(h)$ and $F''(h)$ be the cumulative distribution function of $\bar{h}_{\mathbf{X}}(\xi)$ under distribution $\mathbb{P}'$ and $\mathbb{P}''$, respectively. By the definition of quantile, 
to prove inequality (\ref{eqn:inequ-quantile}), it is sufficient to show that
\[ F'(h) \geq F''(h),\quad \forall h \in \R,
\]
which is equivalent to
\begin{align*} 
    &\mathbb{E}_{\xi\sim\mathbb{P}'}\left[\mathbf{1}(\bar{h}_{\mathbf{X}}(\xi) \leq h)\right] \geq \mathbb{E}_{\xi\sim\mathbb{P}''}\left[\mathbf{1}(\bar{h}_{\mathbf{X}}(\xi) \leq h)\right],\quad \forall h \in \R,
\end{align*}
where $\mathbf{1}(\gamma(\nu, \xi) \leq \gamma)$ is the indicator function. This can be proved in the same way as the proof in \cite{van_parys_data_2021}. As a result, there exists a distribution that attains $\hat{q}_{1 - \epsilon, r,\hat{\mathbb{P}}_S}(\mathbf{X})$ and has support in $\{\xi^j,j\in[S]\}\cup\Xi^*(\mathbf{X})$, which implies the existence of an integer $k\in[S + 1]$ such that
\[ \hat{q}_{1 - \epsilon, r, \hat{\mathbb{P}}_S}(\mathbf{X}) = \bar{h}_{k, \hat{\mathbb{P}}_S} \left( \mathbf{X} \right) . \]

Next, we prove that the integer $k$ does not depend on $\mathbf{X}$ and $\hat{\mathbb{P}}_S$. Let $\mathbb{P}_{\epsilon, r, \hat{\mathbb{P}}_S}$ be the aforementioned worst-case distribution that attains $\hat{q}_{1 - \epsilon, r, \hat{\mathbb{P}}_S}(\mathbf{X})$. Assume without loss of generality that
\[ \bar{h}_{\mathbf{X}}(\xi^1) \leq \cdots \leq \bar{h}_{\mathbf{X}}(\xi^S). \]
Define vector $\mathbf{p}\in\R^{S+1}$ as
\[ \mathbf{p}_j := \mathbb{P}_{\epsilon, r, \hat{\mathbb{P}}_S}(\xi^j),\quad \forall j\in[S],\quad \mathbf{p}_{S+1} := \mathbb{P}_{\epsilon, r, \hat{\mathbb{P}}_S}\left[\Xi^*(\mathbf{X})\right]. \]
Then, by problem (33) in \cite{van_parys_data_2021}, the integer $k$ is the solution to
\begin{align}
\label{eqn:quantile-k}
    \max_{k\in[S], \mathbf{p}\in\R^{S+1}}~ &k\\
    \nonumber\st~ &{\sum}_{j\in[k]} \mathbf{p}_j \leq 1 - \epsilon,~ \mathbf{1}_{S+1}^T \mathbf{p} = 1,~ \mathbf{p} \geq \mathbf{0}_{S + 1},\\
    \nonumber&-\frac1S {\sum}_{j\in[S]} \log(S \mathbf{p}_j) \leq r, 
\end{align}
which is independent of $\mathbf{X}$ and $\hat{\mathbb{P}}_S$. Intuitively, $k$ is the largest integer such that the probability $\mathbb{P}_{\epsilon, r, \hat{\mathbb{P}}_S}$ on the smallest $k$ samples is at most $1 - \epsilon$ and the relative entropy constraint is not violated.
\end{proof}
In the case when $k = S+1$, the evaluation of $\bar{h}_{S+1}(\mathbf{X})$ requires the knowledge of $h_{\mathbf{X}}^*$, which may be unknown in practice. Hence, we focus on the case when $k\in[S]$, which can be guaranteed by choosing suitable values of $\epsilon$ and $r$. Furthermore, the value of $k$ can be computed by solving the convex optimization problem (\ref{eqn:quantile-k}).

Then, we formally define the distributionally robust prescriptor, which is also defined in (\ref{eqn:drpres}).
\begin{definition}[Distributionally Robust Prescriptor]
For all $\epsilon\in(0,1]$ and $r > 0$, the distributionally robust prescriptor $\hat{\mathbf{X}}_{\epsilon, r, \mathbb{P}}$ is a quasi-continuous function of $\mathbb{P}$ that solves
\begin{align}\label{eqn:drpres-1}
    {\min}_{\mathbf{X} \in\R^d}~ g(\mathbf{X}) \quad \st~\hat{q}_{1 - \epsilon, r, \mathbb{P}}(\mathbf{X}) \leq 0.
\end{align}
%
\end{definition}
%
%
%
By Lemma \ref{lem:drpred}, the feasible set of problem (\ref{eqn:drpres-1}) is a subset of $\{ \mathbf{X}\in\R^d ~|~ \bar{h}_k(\mathbf{X}) \leq 0 \}$. Thus, combining with Assumption \ref{asp:dro}, problem (\ref{eqn:drpres}) has a finite optimal value and the distributionally robust prescriptor $\hat{\mathbf{X}}_{\epsilon, r,\hat{\mathbb{P}}_S}$ is well defined. 

\subsubsection{Big-M formulation of problem (\ref{eqn:drpres-mip})}

Now, we provide a mixed-integer reformulation of (\ref{eqn:drpres-mip}) to compute the distributionally robust prescriptor. Choosing $C > 0$ to be a sufficiently large constant, we show that the distributionally robust prescriptor is a solution to
\begin{align}
\label{eqn:drpres-mip-1}
    \min_{\mathbf{X}\in\R^d, \mathbf{b}\in\mathbb{Z}^S}~ &g(\mathbf{X})\\
    \nonumber\st~ &h_\ell(\mathbf{X}, \xi^j) \leq C \mathbf{b}_j,\quad \forall \ell\in[m], ~j \in [S],\\
    \nonumber&\mathbf{1}_S^T \mathbf{b} \leq S - k,\quad \mathbf{b}_j \in \{0, 1\},\quad \forall j \in [S].
\end{align}
Intuitively, the constraints in (\ref{eqn:drpres-mip-1}) enforce the joint chance constraint under the empirical distribution $\hat{\mathbb{P}}_S$. Namely, the constraint $h(\mathbf{X}, \xi^j) \leq \mathbf{0}_m$ is satisfied by at least $k$ samples. In the following theorems, we show that the chance constraint under the true distribution $\mathbb{P}_0$ can also be guaranteed by choosing $k$ to be slightly larger than $(1-\epsilon)S$.
\begin{theorem}
\label{thm:drpres-0}
The solution to (\ref{eqn:drpres-mip-1}) is a distributionally robust prescriptor.
\end{theorem}
\begin{proof}
The formulation (\ref{eqn:drpres-mip-1}) is based on the Big-M method \cite{wolsey_integer_1999}. If the variable $\mathbf{b}_j = 1$, since the constant $C$ is sufficiently large, there is no constraint on $h_\ell(\mathbf{X}, \xi^j)$. Otherwise if the variable $\mathbf{b}_j = 0$, the first constraint becomes
\[ h_\ell(\mathbf{X}, \xi^j) \leq 0,\quad \forall \ell\in[m], \]
which is equivalent to the condition $\bar{h}(\mathbf{X}, \xi^j) \leq 0$. With a given $\mathbf{X}\in\R^d$, the constraint $\mathbf{1}_S^T \mathbf{b} \leq S - k$ requires that the above condition hold for at least $k$ samples. To achieve the minimum over $\mathbf{X}$, the condition $\mathbf{b}_j = 0$ should hold for the $k$ indices that correspond to the $k$ smallest values in $\{\bar{h}_{\mathbf{X}}(\xi^j), j\in[S]\}$. In other words, the constraints in (\ref{eqn:drpres-mip-1}) are equivalent to
\[ \bar{h}_k(\mathbf{X}) \leq 0. \]
Combining with Lemma \ref{lem:drpred}, we get the desired result.
\end{proof}
In the case of the chance-constrained DC OPF problem in Section \ref{sec:dc-ccopf}, problem (\ref{eqn:drpres-mip-1}) becomes a MIP problem. On the other hand, for the chance-constrained AC OPF problem introduced in Section \ref{sec:ac-ccopf}, problem (\ref{eqn:drpres-mip-1}) is equivalent to a mixed-integer QCQP. In Section \ref{sec:relaxation}, we apply the convex relaxation to the QCQP and when the relaxation is exact, problem (\ref{eqn:drpres-mip-1}) is equivalent to a C-MINLP problem, which can be handled by off-the-shelf convex optimization solvers.

\subsubsection{Proof of Theorem \ref{thm:drpres-1}}

\begin{proof}[Proof of Theorem \ref{thm:drpres-1}]
By the definition of the prescriptor $\hat{\mathbf{X}}_{\epsilon, r, \hat{\mathbb{P}}_S}$, we have
\begin{align*}
    \hat{q}_{1 - \epsilon, r, \hat{\mathbb{P}}_S}\left(\hat{\mathbf{X}}_{\epsilon, r, \hat{\mathbb{P}}_S}\right) \leq 0.
\end{align*}
By a similar technique to the proof of Lemma \ref{lem:drpred}, the results of Theorem 10 of \cite{van_parys_data_2021} also holds for the predictor $\hat{q}_{1-\epsilon, r, \hat{\mathbb{P}}_S}$ and we have
\begin{align*}
    \limsup_{S\rightarrow\infty}\frac1S \log\mathbb{P}_\infty\Big[ &\hat{q}_{1-\epsilon,r,\hat{\mathbb{P}}_S}\left(\hat{\mathbf{X}}_{\epsilon, r, \hat{\mathbb{P}}_S}\right)\\
    &\hspace{2em}< q_{1 - \epsilon} \left[\bar{h}(\hat{\mathbf{X}}_{\epsilon, r, \hat{\mathbb{P}}_S},\cdot), \mathbb{P}_0\right] \Big] \leq -r.
\end{align*}
Combining the above two inequalities, we get
\begin{align*}
    \mathbb{P}_\infty\left[ q_{1 - \epsilon} \left[\bar{h}(\hat{\mathbf{X}}_{\epsilon, r, \hat{\mathbb{P}}_S},\cdot), \mathbb{P}_0\right] \leq 0 \right] \geq 1 - \exp\left[-rS + o(S)\right].
\end{align*}
By the definition of the quantile and applying the union bound, it follows that
\begin{align*}
    \mathbb{P}_\infty\Big[ h_\ell\left(\hat{\mathbf{X}}_{\epsilon, r, \hat{\mathbb{P}}_S}, \xi\right) \leq 0,\quad &\forall \ell \in [m] \Big]\\
    &\geq 1 - \epsilon - \exp\left[ -rS + o(S) \right].
\end{align*}
which is the desired result of this theorem.
\end{proof}

\subsubsection{Proof of Theorem \ref{thm:drpres-2}}

\begin{proof}[Proof of Theorem \ref{thm:drpres-2}]
We first construct a set where the distributionally robust predictor $\hat{q}_{1-\epsilon, r, S}$ takes a positive value. Assume conversely that 
\[ {p}_S \coloneqq \mathbb{P}_\infty\left[ g\left(\tilde{\mathbf{X}}_{\epsilon, r, \hat{\mathbb{P}}_S}\right) < g \left(\hat{\mathbf{X}}_{\epsilon, r, \hat{\mathbb{P}}_S} \right) \right] > 0. \]
%
Since the prescriptor $\hat{\mathbf{X}}_{1-\epsilon, r, S}$ attains the minimal objective value under the constraint $\bar{h}_{k}( \mathbf{X} ) \leq 0$, we have
\[  \mathbb{P}_\infty\left[ \bar{h}_{k}\left( \tilde{\mathbf{X}}_{\epsilon, r, \hat{\mathbb{P}}_S} \right) > 0 \right] \geq p_S. \]
Since $\mathbb{P}_\infty(\mathbf{b} > z)$ is a right-continuous function of $z\in\R$ for every random variable $\mathbf{b}$, there exists a sufficiently small constant $\tau > 0$ such that
\[  \mathbb{P}_\infty\left[ \bar{h}_{k}\left( \tilde{\mathbf{X}}_{\epsilon, r, \hat{\mathbb{P}}_S} \right) > \tau \right] \geq p_S / 2 > 0. \]
Consider the set
\begin{align*}
    \mathcal{X}_S \coloneqq \Big\{ \left( \tilde{\mathbf{X}}_{\epsilon, r, \hat{\mathbb{P}}_S}, \hat{\mathbb{P}}_S \right)  ~|~ \bar{h}_{k}\left( \tilde{\mathbf{X}}_{\epsilon, r, \hat{\mathbb{P}}_S} \right) > \tau \Big\} \subset \R^d \times \mathcal{P}.
\end{align*}
Since $\tilde{\mathbf{X}}_{\epsilon, r, \hat{\mathbb{P}}_S}$ is a quasi-continuous function of the empirical distribution $\hat{\mathbb{P}}_S$, the set $\mathcal{X}_S$ is a non-empty quasi-open set \cite[Prop. 1.2.4]{neubrunn_quasi-continuity_1988} under the product topology of the Euclidean topology on $\R^d$ and the weak topology on $\mathcal{P}$. Therefore, the interior of $\mathcal{X}_S$, denoted as $\mathcal{X}_S^\circ$, is non-empty.

Now, we construct a data-driven predictor $\tilde{q}_{1-\epsilon, r, \mathbb{P}}$ that is continuous and does not dominate the distributionally robust predictor $\hat{q}_{1-\epsilon, r, \mathbb{P}}$. For every point $(\mathbf{X}, \mathbb{P}) \in \mathcal{X}_S$, we define
\[ \mathrm{d}(\mathbf{X}, \mathbb{P}) \coloneqq \min\left\{ \mathrm{dist}\left[ (\mathbf{X}, \mathbb{P}), \mathcal{X}_S^c \right] , \tau \right\}, \]
where $\mathcal{X}_S^c \coloneqq (\R^d \times \mathcal{P}) \backslash \mathcal{X}_S$ is the complementary set of $\mathcal{X}_S$ and the distance function is induced by the Euclidean $2$-norm on $\R^d$ and the Prokhorov metric \cite{prokhorov_convergence_1956} on $\mathcal{P}$. Since the distance function is continuous, the function $\mathrm{d}(\cdot,\cdot)$ is also continuous and takes positive values on $\mathcal{X}_S^\circ$. We define
\begin{align*}
    \tilde{q}_{1-\epsilon, r, \mathbb{P}}(\mathbf{X}) \coloneqq 
    \hat{q}_{1-\epsilon, r, \mathbb{P}}(\mathbf{X}) - \mathrm{d}(\mathbf{X}, \mathbb{P}),\quad \forall (\mathbf{X}, \mathbb{P}) \in \R^d \times \mathcal{P}.
\end{align*}
It follows from the definition of $\mathrm{d}$ and $\mathcal{X}_S$ that
\begin{align}
\label{eqn:drpres-proof-4}
    0 \leq \tilde{q}_{1-\epsilon, r, \mathbb{P}}(\mathbf{X}) \leq \hat{q}_{1-\epsilon, r, \mathbb{P}}(\mathbf{X}) ,\quad &\forall (\mathbf{X}, \mathbb{P}) \in \mathcal{X}_S,
\end{align}
where the second inequality holds strictly on $\mathcal{X}_S^\circ$. Note that the predictor $\tilde{q}_{1-\epsilon, r, \hat{\mathbb{P}}_S}$ is a data-driven predictor since it only relies on the empirical distribution $\hat{\mathbb{P}}_S$.  

Finally, we show that $\tilde{q}_{1-\epsilon, r, \mathbb{P}}$ is feasible for problem (5) in \cite{van_parys_data_2021}, namely,
\begin{align}
\label{eqn:drpres-proof-3}
    \limsup_{S\rightarrow \infty} \frac1S \log\mathbb{P}_\infty\left[ \tilde{q}_{1-\epsilon, r, \hat{\mathbb{P}}_S}(\mathbf{X}) < q_{1-\epsilon}\left[\bar{h}(\mathbf{X},\cdot), \mathbb{P}_0\right] \right] \leq -r.
\end{align}
Since condition (\ref{eqn:drpres-proof-3}) is satisfied by $\hat{q}_{1-\epsilon, r, \hat{\mathbb{P}}_S}$ and 
\[ \tilde{q}_{1-\epsilon, r, \hat{\mathbb{P}}_S}(\mathbf{X}) = \hat{q}_{1-\epsilon, r, \hat{\mathbb{P}}_S}(\mathbf{X}),\quad \forall \mathbf{X} \in\R^d ~\st ~\mathbf{X} \neq \tilde{\mathbf{X}}_{\epsilon, r, \hat{\mathbb{P}}_S}, \]
we only need to show
\begin{align}
\label{eqn:drpres-proof-5}
    \limsup_{S\rightarrow\infty}\frac1S \log\mathbb{P}_\infty\Big[ &\tilde{q}_{1-\epsilon,r,\hat{\mathbb{P}}_S}\left(\hat{\mathbf{X}}_{\epsilon, r, \hat{\mathbb{P}}_S}\right)\\
    \nonumber&\hspace{2em}< q_{1 - \epsilon} \left[\bar{h}(\tilde{\mathbf{X}}_{\epsilon, r, \hat{\mathbb{P}}_S},\cdot), \mathbb{P}_0\right] \Big] \leq -r.
\end{align}
Since prescriptor $\tilde{\mathbf{X}}_{\epsilon, r, \hat{\mathbb{P}}_S}$ satisfies condition (\ref{eqn:drpres-bound-1}), it holds that
\begin{align*}
    \limsup_{S\rightarrow\infty}\frac1S \log\mathbb{P}_\infty\Big[ q_{1 - \epsilon} \left[\bar{h}(\tilde{\mathbf{X}}_{\epsilon, r, \hat{\mathbb{P}}_S},\cdot), \mathbb{P}_0\right] < 0 \Big] \leq -r.
\end{align*}
Combining with property (\ref{eqn:drpres-proof-4}), we get the desired result (\ref{eqn:drpres-proof-5}). 

In summary, we have constructed a predictor $\tilde{q}_{1-\epsilon,r,\hat{\mathbb{P}}_S}\left(\hat{\mathbf{X}}_{\epsilon, r, \mathbb{P}}\right)$ that is continuous and feasible for problem (5) in \cite{van_parys_data_2021}, but it does not dominate the distributionally robust predictor $\hat{q}_{1-\epsilon,r,\hat{\mathbb{P}}_S}\left(\hat{\mathbf{X}}_{\epsilon, r, \mathbb{P}}\right)$. However, this is in contradiction with Theorem 10 in \cite{van_parys_data_2021}, which claims that the distributionally robust predictor is the strong solution to problem (5).
\end{proof}

\end{appendix}

\end{document}